\documentclass[10pt,a4paper]{amsart}
\usepackage{amsmath,amssymb, amsbsy}
\usepackage{color,psfrag}
\usepackage[dvips]{graphicx}

\usepackage{color}
\usepackage{mathtools}
\usepackage{hyperref}
\theoremstyle{plain}
\usepackage{enumerate}
\newtheorem{theorem}{Theorem}[section]

\newtheorem{lemma}{Lemma}[section]
\newtheorem{corollary}{Corollary}[section]
\newtheorem{remark}{Remark}[section]

\newcommand{\rn}{\mathbb{R}^{N}}

\newcommand{\hn}{\mathbb{H}^{N}}
\newcommand{\rgt}{\bigg(\frac{\partial u}{\partial r}\bigg)^2}

\newcommand{\rg}{\frac{\partial u}{\partial r}}

\numberwithin{equation}{section} \allowdisplaybreaks

\usepackage[text={6in,8.6in},centering]{geometry}
\begin{document}
	
	\title[ Poincar\'e inequalities on Riemannian models]{On some strong Poincar\'e inequalities\\ on Riemannian models and their improvements}
	
	\author[Elvise BERCHIO]{Elvise BERCHIO}
\address{\hbox{\parbox{5.7in}{\medskip\noindent{Dipartimento di Scienze Matematiche, \\
Politecnico di Torino,\\
        Corso Duca degli Abruzzi 24, 10129 Torino, Italy. \\[3pt]
        \em{E-mail address: }{\tt elvise.berchio@polito.it}}}}}

\author[Debdip GANGULY]{Debdip GANGULY}
\address{\hbox{\parbox{5.7in}{\medskip\noindent{Department of Mathematics,\\
 Indian Institute of Technology Delhi,\\
 IIT Campus, Hauz Khas, New Delhi\\
        Delhi 110016, India. \\[3pt]
        \em{E-mail address: }{\tt debdipmath@gmail.com}}}}}
	
	\author[Prasun Roychowdhury]{Prasun Roychowdhury}
	\address{\hbox{\parbox{5.7in}{\medskip\noindent{Department of Mathematics,\\
 Indian Institute of Science Education and Research,\\
 Dr. Homi Bhabha Road, Pashan\\
        Pune 411008, India. \\[3pt]
        \em{E-mail address: }{\tt prasunroychowdhury1994@gmail.com}}}}}


	\maketitle
	
	\begin{abstract}
		We prove second and fourth order improved Poincar\'e type inequalities on the hyperbolic space involving Hardy-type 
		remainder terms. Since theirs l.h.s. only involve the radial part of the gradient or of the laplacian, they can be seen as stronger versions of the classical Poincar\'e inequality. We show that such inequalities hold true on model manifolds as well, under suitable curvature assumptions and 
		sharpness of some constants is also discussed. 	
	\end{abstract}
	
	\section{Introduction}

	Let $M$ be a Cartan-Hadamard manifold of dimension $N$ (namely, a manifold which is complete, simply-connected, and has everywhere non-positive sectional curvature). It is well known that if the sectional curvatures of $M$ are bounded above by a \it strictly negative \rm constant, then $M$ admits a Poincar\'e type, or $L^2$-gap, inequality, namely there exists $\Lambda>0$ such that
	\begin{equation*}
		\int_M|\nabla_g u|^2\,{\rm d}v_g\ge \Lambda\int_Mu^2\,{\rm d}v_g\ \ \ \forall u\in C_c^\infty(M).
	\end{equation*}

	Recently, there has been a constant effort to improve the Poincar\'e inequality in terms of adding possibly \emph{optimal} Hardy weights, i.e. adding on the r.h.s. a term of the form $\int_M W\,u^2\,{\rm d}v_g$ with $W\geq 0$ ``as large as possible", see \cite{pinch} for a general treatment of Hardy weights for second-order elliptic operators. Starting 
	from the works \cite{AK} and \cite{EDG}, where a Poincar\'e-Hardy inequality was shown with sharp constants, further 
	generalisation to $p$-Laplacian and higher order case have been obtained in \cite{BAGG} and \cite{BG1}, respectively.  
	We also refer to \cite{BGGP} for more general improvements and the study of extremals. The kind of weights obtained in these papers, which are singular at a fixed point of $M$, makes this subject a sort of lateral branch of that very rich field of research originated from the seminal paper \cite{Brezis}, and dealing with possible improvements (not only of $L^2$ type) of the classical Hardy inequality on bounded euclidean domains or on curved spaces, see e.g. \cite{ANR, GFT, Carron, Dambrosio, Kombe1, Kombe2, AKR,KS, Yang} and reference therein.

	\medskip 
	The most studied example of $M$ is clearly the \it hyperbolic space \rm ${\mathbb H}^N$, where $\Lambda=\left( \frac{N-1}{2} \right)^{2}$. In this setting the following improved Poincar\`e inequality was established in \cite{EDG} for all $u\in C_c^\infty(\hn)$:
	\begin{align}
	\label{poin_hardy_sharp}
	& \int_{\hn} |\nabla_{\hn} u|^2 \,  {\rm d}v_{\hn} \geq \left( \frac{N-1}{2} \right)^{2} \int_{\hn} u^2 \, {\rm d}v_{\hn}\\ \notag
	&+ \frac{1}{4} \int_{\hn} \frac{u^2}{r^2} \, {\rm d}v_{\hn}+  \frac{(N-1)(N-3)}{4} \int_{\hn} \frac{u^2}{\sinh^2 r} \ dv_{\hn}\,,
	\end{align}
		where $N\geq 3$, $r:=\rho(x,o)$ denotes the geodesic distance function and $o\in \hn$ is fixed. Moreover, all constants in \eqref{poin_hardy_sharp} are proved to be sharp in a suitable sense.

		 In view of the work \cite{VHN}, where a similar question has been recently posed in the context of Hardy inequalities, one may wonder whether  inequality  \eqref{poin_hardy_sharp} still holds if we replace $ |\nabla_{\hn} u|^2$ with its radial part $(\frac{\partial u}{\partial r})^2$. Since, by Gauss's lemma one has $|\nabla_{\hn} u|^2\geq (\frac{\partial u}{\partial r})^2$, giving an answer to this question is by no means obvious and it represents the first goal of this paper. Indeed, we prove that the following \emph{stronger} version of \eqref{poin_hardy_sharp} holds
		\begin{align}\label{poin_hardy_sharp2}
	& \int_{\hn} \rgt \ {\rm d}v_{\hn} \geq \left( \frac{N-1}{2} \right)^{2} \int_{\hn} u^2 \, {\rm d}v_{\hn}\\ \notag
	&+ \frac{1}{4} \int_{\hn} \frac{u^2}{r^2} \, {\rm d}v_{\hn}+  \frac{(N-1)(N-3)}{4} \int_{\hn} \frac{u^2}{\sinh^2 r} \ dv_{\hn}\,,
	\end{align}	
	for all $u\in C_c^\infty(\hn)$ and $N\geq 3$. Clearly, \eqref{poin_hardy_sharp2} reproves inequality \eqref{poin_hardy_sharp} but we derive it by using a different technique: here the main tool exploited is spherical harmonics decomposition technique, while the proof of \eqref{poin_hardy_sharp} was based on finding a ground state and on criticality theory, see \cite[Theorem~2.1]{EDG}. Furthermore, we argue quite differently also in proving the optimality of the constants appearing in \eqref{poin_hardy_sharp2}. Notice that the sharpness of all constants in  \eqref{poin_hardy_sharp2} can be derived by combining Gauss's lemma with the sharpness of the corresponding constants in \eqref{poin_hardy_sharp}. Nevertheless, in this paper we provide and alternative, and more direct, proof of the sharpness of the dominating term at infinity of  \eqref{poin_hardy_sharp2}, namely of the constant $\frac{1}{4}$, which is based on the delicate construction of a suitable minimising sequence. This argument may have its own interest in the study of related partial differential equations, furthermore it can be carried over to more general Riemannian models having negative sectional curvatures bounded above, see Theorem \ref{mainthm-1} below.

	\medskip
	It is known from \cite{SK} that in the Hyperbolic space there also hold the following second order Poincar\`e inequalities:
	\begin{align}\label{higher_order_poin}
	  \int_{\hn}(\Delta_{\hn} u)^2 {\rm d}v_{\hn}\geq \bigg(\frac{N-1}{2}\bigg)^{2(2-l)} \int_{\hn} |\nabla_{\hn}^l u|^2{\rm d}v_{\hn},
	\end{align}
	for all $u\in C_c^\infty(\hn)$   with $N\geq 3$  and with $l=0$ or $l=1$. We refer to \cite{ngo} for a proof of the sharpness of the above constants. Now, in view of \eqref{higher_order_poin} and for what previously stated in the first order, it is natural to think about possible extensions of  \eqref{poin_hardy_sharp2} to the second order. In this respect the following inequality from \cite[Theorem 5.2]{VHN} turns out to be meaningful:
    \begin{align}\label{base}
 \int_{\mathbb{H}^{N}} (\Delta_{\hn} u)^2  \ dv_{\hn} \geq  \int_{\mathbb{H}^{N}} (\Delta_{r,\hn} u)^2  \ dv_{\hn} 
\end{align}
for all $u\in C_c^\infty(\hn)$  with $N\geq 5$,  where $\Delta_{r,\hn}$ denotes the radial part of the Laplace-Beltrami operator $\Delta_{\hn}$ on $\hn$, see also Lemma \ref{rad_lap} below. Clearly, inequality \eqref{base} suggests that a possible stronger version of \eqref{higher_order_poin} might involve the operator $\Delta_{r,\hn}$.  Motivated by this, in the present paper we prove the following improved Poincar\`e type inequality:
\begin{align}\label{PR2} 
&
    \int_{\hn} (\Delta_{r,\hn} u)^2 \ {\rm d}v_{\hn} \geq \left( \frac{N-1}{2} \right)^{2} \int_{\hn} \rgt \ {\rm d}v_{\hn}  \\ \notag  
    &+\frac{1}{4} \int_{\hn} \frac{1}{r^2}\rgt \ {\rm d}v_{\hn} +\frac{(N^2- 1)}{4}\int_{\hn} \frac{1}{\sinh^2 r} \rgt \ {\rm d}v_{\hn}
     \end{align}
     for all $u\in C_c^\infty(\hn)$   with $N\geq 3$.
     
We notice that the interest of \eqref{PR2} is not only due to the fact that it looks likes the proper second order analogue of \eqref{poin_hardy_sharp2}, but also to the fact that it improves \eqref{higher_order_poin} from several points of view.  In fact, on one hand, by combining \eqref{poin_hardy_sharp2} and a related weighted inequality, from \eqref{PR2} we derive the inequality:
\begin{align}\label{PR4} 
&
    \int_{\hn} (\Delta_{r,\hn} u)^2 \ {\rm d}v_{\hn} \geq \left( \frac{N-1}{2} \right)^{4} \int_{\hn} u^2 \ {\rm d}v_{\hn}  \\ \notag  
    &+\left(\frac{N-4}{4}\right)^2 \int_{\hn} \frac{u^2}{r^4} \ {\rm d}v_{\hn} +\frac{(N^2- 1)}{16}\int_{\hn} \frac{u^2}{r^2} \ {\rm d}v_{\hn}\,,
     \end{align}
	 for all $u\in C_c^\infty(\hn)$  with $N\geq 5$;   in view of \eqref{base}, \eqref{PR4} reads as a \emph{strong} (and improved) version of \eqref{higher_order_poin} with $l=0$. On the other hand, a clever exploitation of \eqref{PR2}, jointly with the spherical harmonics decomposition technique, yields the following improved version of \eqref{higher_order_poin} with $l=1$:
	 \begin{align}\label{PR5}
 &   \int_{\hn} (\Delta_{\hn} u)^2 \ {\rm d}v_{\hn} \geq  \left( \frac{N-1}{2} \right)^{2} \int_{\hn} |\nabla_{\hn} u|^2 \ {\rm d}v_{\hn} \\ \notag
     &+ \frac{1}{4} \int_{\hn} \frac{|\nabla_{\hn} u|^2}{r^2} \ {\rm d}v_{\hn} 
     +\frac{(N^2- 1)}{4}\int_{\hn} 
    \frac{|\nabla_{\hn} u|^2}{\sinh^2 r} \,  {\rm d}v_{\hn}.
    \end{align}
	 for all $u\in C_c^\infty(\hn)$  with $N\geq 5$. We note that, despite their similarity, \eqref{PR2} and \eqref{PR5} are not mutually implicated. 	 	 
	 
	We complete this short discussion by mentioning that inequality \eqref{PR4} can be compared with the following improved Poincar\`e inequality proved in \cite{EDG}:
	\begin{align}\label{PR} 
& \int_{\mathbb{H}^{N}} (\Delta_{\hn} u)^2  \ dv_{\hn} \geq   \left( \frac{N-1}{2} \right)^{4}\int_{\hn} u^2 \ dv_{\hn} \\ \notag
 &+ \frac{9}{16} \int_{\mathbb{H}^{N}} \frac{u^2}{r^4} \ dv_{\hn}+ \frac{(N-1)^2}{8} \int_{\mathbb{H}^{N}} \frac{u^2}{r^2} \ dv_{\hn}  \,,
\end{align}
for all $u\in C_c^\infty(\hn)$  with $N\geq 5$  and where the constant $\frac{(N-1)^2}{8}$ was proved to be sharp. In view of \eqref{base},  \eqref{PR4} has a l.h.s. with a stronger impact than that of \eqref{PR}, furthermore, since $\frac{9}{16}<\frac{(N-4)^2}{16}$ for $N\geq 8$, inequality \eqref{PR4} gives a better improvement than \eqref{PR} of the Poincar\`e inequality \eqref{higher_order_poin} near the pole, when $N\geq 8$. However, for all dimensions, the r.h.s. of \eqref{PR} performs better at infinity. Also we mention that, unfortunately, neither \eqref{PR4} or \eqref{PR} solve the problem of sharpness of the constant in front of $ \frac{1}{r^4}$ which is still open.

    \medskip
 We notice that within the paper we prove and state all results on general model manifolds satisfying suitable curvature bounds and having $\hn$ as a remarkable particular case, see Section 2 below.

    \medskip
    The paper is organized as follows: we state our main results in Section \ref{2} and Section \ref{prototypesMM}, while the rest of the paper deals with the proofs of the results. More precisely,  in Section \ref{2} we first briefly recall some useful facts concerning to Riemannian model manifolds and we introduce some curvature bounds related to the models on which we will settle our analysis. Then, we state our first and second order improved Poincar\'e inequalities. In Section \ref{prototypesMM} we discuss the first order results on suitable prototype model manifolds allowing for unbounded curvatures from below. In Section \ref{3} we give the proof of our improved strong Poincar\'e inequality in the first order, namely inequality \eqref{poin_hardy_sharp2}  but on more general Riemannian models. Section \ref{4} contains the proof of inequalities related to the radial part of Laplace-Beltrami operator and to the radial part of gradient, namely the analogous of inequality \eqref{PR2}  on model manifolds. In Section \ref{5} we first prove \eqref{base} on general models and then we derive our improved strong Poincar\'e inequality in the second order, which reads as \eqref{PR4} on $\hn$. Finally, in Section \ref{6} we derive our improved version of \eqref{higher_order_poin} with $l=1$ on model manifolds, i.e. \eqref{PR5} on $\hn$.

    \medskip
\section{Preliminaries and Statement of Main results}\label{2}

    This section is devoted to the statements of the main results. Before stating our theorems let us briefly recall some 
	of the known facts concerning Riemannian model manifolds. 
	
	\medskip
	
	An $N$-dimensional \it Riemannian model\rm $\,(M,g)$ is an $N$-dimensional Riemannian manifold admitting a pole  $o\in M$ and whose metric $g$ is given in spherical coordinates by
	\begin{equation}\label{meetric}
		{\rm d}s^2 = {\rm d}r^2 + \psi^2(r) \, {\rm d}\omega^2,
	\end{equation}
	where $d\omega^2$ is the metric on sphere $\mathbb{S}^{N-1}$ and $\psi$ satisfies
	\begin{align}\label{psi}
	&\psi \text{ is a } C^{\infty} \text{ nonnegative function on } [0,+\infty),  \text{ positive on } (0,+\infty) \\ \notag
	 &\text{such that } \psi'(0) =1 \text{ and } \psi^{(2k)}(0)= 0 \text{ for all } k\geq 0\,.
	\end{align}
These conditions on $\psi$ ensure that the manifold is smooth and
the metric at the pole $o$ is given by the euclidean metric
\cite[Chapter 1, 3.4]{PP}. The coordinate $r$, by construction, represents the Riemannian distance from the pole $o,$ see e.g. \cite{RGR,AG,PP} for further details. In particular, all the assumptions above are satisfied by $\psi(r)=r$ and by $\psi(r)=\sinh(r)$: in the first case $M$ coincides with the euclidean space $\rn$, in the latter with the hyperbolic space $\hn$.\par
 
  It is known that there exist an orthonormal
	frame $\{F_{j} \}_{j = 1, \ldots, N}$ on $(M,g)$ where
	$F_{N}$ corresponds to the radial coordinate, and $F_{1}, \ldots, F_{N-1}$ to the spherical coordinates, for which
	$F_{i} \wedge F_{j}$ diagonalize the curvature operator $\mathcal{R}$ :
	
	\[
	\mathcal{R} (F_{i} \wedge F_{N}) = - \frac{\psi^{\prime \prime}}{\psi} F_{i} \wedge F_{N}, \quad i < N,
	\]
	\[
	\mathcal{R} (F_{i} \wedge F_{j}) = - \frac{(\psi^{\prime})^2 - 1}{\psi^2} F_{i} \wedge F_{j}, \quad i,j < N.
	\]

	The quantities
	\begin{equation}\label{curvature}
		K_{\pi,r}^{rad} := - \frac{\psi^{\prime \prime}}{\psi}\quad \text{and}  \quad H_{\pi,r}^{tan} := - \frac{(\psi^{\prime})^2 - 1}{\psi^{2}}
	\end{equation}
	then coincide with the sectional curvatures w.r.t. planes containing the radial direction and, respectively, orthogonal to it. In what follows we will sometimes need to assume that
	\begin{equation}\label{con2}
			K_{\pi,r}^{rad} \leq -1 \quad \text{in } (0,+\infty)\,.
		\end{equation}
	Since, by the Sturm-Comparison Theorem, the above condition also implies that $H_{\pi,r}^{tan} \leq -1$, with \eqref{con2} we basically require the boundedness from above of both sectional curvatures. See also Lemma \ref{sturm} in the following. 
	Finally, a further crucial quantity in our statements will be
	\begin{equation}\label{lambdarad}
		\Lambda_{\pi,r}^{rad} := - 2 K_{\pi,r}^{rad} - (N-3) H_{\pi,r}^{tan}\,,
	\end{equation}
	
	$\Lambda_{\pi,r}^{rad}$ is related to the bottom of the spectrum of the laplacian, see inequality \eqref{eq_mainthm-1} below; in particular, when $\psi(r)=\sinh(r)$, then $\Lambda_{\pi,r}^{rad}=N-1$ hence $\frac{(N-1)}{4} \Lambda_{\pi,r}^{rad}$ coincides with the bottom of the spectrum of the laplacian in ${\mathbb H}^N$.

	 \medskip
	 
Next, we recall that the Riemannian Laplacian of a scalar function $u$ on $M$ is given by
\begin{equation}\label{laplacian}
\Delta_{g} u (r, \theta_{1}, \ldots, \theta_{N-1})  =
\frac{1}{\psi^2} \frac{\partial}{\partial r} \left[ (\psi(r))^{N-1} \frac{\partial u}{\partial r}(r, \theta_{1},
\ldots, \theta_{N-1}) \right] \\
+ \frac{1}{\psi^2} \Delta_{\mathbb{S}^{N-1}} u(r, \theta_{1}, \ldots, \theta_{N-1}),
\end{equation}
where $\Delta_{\mathbb{S}^{N-1}}$ is the Riemannian Laplacian on the unit sphere $\mathbb{S}^{N-1}.$ In particular,
for radial functions, namely functions depending only on $r$, $\Delta_{g} u$ reads
\begin{equation}\label{radlaplacian}
\Delta_{r,g} u(r) = \frac{1}{(\psi(r))^{N-1}} \frac{\partial}{\partial r} \left[ (\psi(r))^{N-1} \frac{\partial u}{\partial r}
(r) \right] =  u^{\prime \prime}(r) + (N-1) \frac{\psi^{\prime}(r)}{\psi(r)} u^{\prime}(r),
\end{equation}
where from now on a prime will denote, for radial functions, derivative w.r.t $r$.

    \medskip
	We are now ready to state our results.
	
	\begin{theorem}\label{mainthm-1}
	 Let $\,(M,g)$ be an $N$-dimensional Riemannian model with $N\geq 3$ and with metric $g$ as given in \eqref{meetric} with $\psi$ satisfying \eqref{psi}.
Then, for all $ u \in C_c^{\infty}(M\setminus\{0\})$ there holds
		\begin{equation}\label{eq_mainthm-1}
			\begin{split}
				\int_{M} \rgt \, {\rm d}v_{g} \,- \,
				\frac{(N-1)}{4} \int_{M} \Lambda_{\pi,r}^{rad} \, u^2 \, {\rm d}v_{g}
				\geq \frac{1}{4} \int_{M} \frac{u^2}{r^2} \, {\rm d}v_{g} + \frac{(N-1)(N-3)}{4} \int_{M} \frac{u^2}{\psi^2} \, {\rm d}v_{g} \,,
			\end{split}
		\end{equation}
		with $\Lambda_{\pi,r}^{rad} $ as defined in \eqref{lambdarad}. Moreover, if condition \eqref{con2} holds and furthermore
			\begin{equation}\label{con2bis}
			\frac{\psi^{\prime}}{\psi}\sim C \, r^{a} \quad \text{as } r \rightarrow +\infty\,
		\end{equation}
		for some $C>0$ and $a\geq 0$, then the constant $\frac{1}{4}$ in \eqref{eq_mainthm-1} is sharp, i.e.
		\begin{equation}\label{ratio}
			\frac{1}{4} = {\inf}_{H^1(M)\setminus\{0\}}\:\frac{\int_{M}\big(\frac{\partial u}{\partial r}\big)^2 \ {\rm d}v_{g} - \frac{(N-1)}{4}\int_{M} \Lambda_{\pi,r}^{rad}  \,u^2 \ {\rm d}v_{g}}{\int_{M}\frac{u^2}{r^2} \ {\rm d}v_g}.
		\end{equation}
	 \end{theorem}
	 
	 \begin{remark}
		A couple of remarks are in order about the further conditions required in the second part of the statement of Theorem \ref{mainthm-1}. Condition \eqref{con2} seems by no means technical as suggested by the following simple example. Consider the euclidean space, then $\psi(r)=r$ and \eqref{con2} clearly fails. On the other hand, in this case, $\Lambda_{\pi,r}^{rad} \equiv 0$ and inequality \eqref{eq_mainthm-1} becomes the (strong) Hardy inequality:
		$$
		\int_{\rn} \rgt \, {\rm d}v_{\rn} \geq  \frac{(N-2)^2}{4} \int_{\rn} \frac{u^2}{r^2} \, {\rm d}v_{\rn}\,,
			$$
	hence $1/4$ is no more the sharp constant in front of the term $\frac{u^2}{r^2}$. As concerns condition \eqref{con2bis}, it is needed to show that the quotient \eqref{ratio} is finite for our minimizing sequence, see Section \ref{3}. Nevertheless, we notice that this condition is not ``too restrictive", in the sense that it allows unbounded curvatures from below as it happens, for instance, if $\psi(r)=r \, e^{r^2}$ for which \eqref{con2bis} holds with $C=2$ and $a=1$; see also Section \ref{prototypesMM} for further examples and the related disussion.
	\end{remark}

	We point out that when \eqref{con2} holds inequality \eqref{eq_mainthm-1} implies the following more explicit inequality:

\begin{corollary}\label{use_cor_1}
	 Let $\,(M,g)$ be an $N$-dimensional Riemannian model with $N\geq 3$ and with metric $g$ as given in \eqref{meetric} with $\psi$ satisfying \eqref{psi} and \eqref{con2}.
	Then, for all $u \in C_c^{\infty}(M),$ there holds
	\begin{equation}\label{eqn_use_cor_1}
	\int_{M} \rgt \ {\rm d}v_g -   \left( \frac{N-1}{2} \right)^{2} \int_{M} u^2 \ {\rm d}v_g  \\  \geq \frac{1}{4} \int_{M} \frac{u^2}{r^2} \ {\rm d}v_g +\frac{(N-1)(N-3)}{4}\int_{M} \frac{u^2}{\psi^2} \ {\rm d}v_g\,.
	\end{equation}
	
\end{corollary}
A remarkable particular case to which Theorem \ref{mainthm-1} applies is when $M=\hn$, i.e. $\psi(r)=\sinh(r)$. In this case all constants in \eqref{eq_mainthm-1} are proved to be sharp.

\begin{corollary}\label{use_cor_2}
	 Let $M=\hn$, the Hyperbolic space with $N\geq 3$. 	Then, for all $u \in C_c^{\infty}(\hn)$ there holds
	  \begin{equation*}
 \int_{\hn} \rgt \ {\rm d}v_{\hn} -  \left( \frac{N-1}{2} \right)^{2} \int_{\hn} u^2 \ {\rm d}v_{\hn}  \\  \geq \frac{1}{4} \int_{\hn} \frac{u^2}{r^2} \ {\rm d}v_{\hn} +\frac{(N-1)(N-3)}{4}\int_{\hn} \frac{u^2}{(\sinh r)^2} \ {\rm d}v_{\hn}
 \end{equation*}
with all constants sharp. More precisely, the Poincar\'e 
 constant $\left( \frac{N-1}{2} \right)^{2}$ is sharp in the sense that no inequality of the form 
	$$
	\int_{\hn} \rgt \  {\rm d}v_{\hn}  \geq  c  \int_{\hn}  u^2 \  {\rm d}v_{\hn}
	$$
	holds, for  $ u \in C_{c}^{\infty}(\hn)$, when $c > \left( \frac{N-1}{2} \right)^{2}$. The constant $\frac{1}{4}$ is sharp in the sense explained in Theorem \ref{mainthm-1} while the constant $\frac{(N-1)(N-3)}{4}$ is sharp in the sense that no inequality of the form 
	$$
	 \int_{\hn} \rgt \ {\rm d}v_{\hn} -  \left( \frac{N-1}{2} \right)^{2} \int_{\hn} u^2 \ {\rm d}v_{\hn} - \frac{1}{4} \int_{\hn} \frac{u^2}{r^2} \ {\rm d}v_{\hn} \geq c\int_{\hn} \frac{u^2}{(\sinh r)^2} \ {\rm d}v_{\hn}
	$$
	holds, for  $ u \in C_{c}^{\infty}(\hn)$, when $c > \frac{(N-1)(N-3)}{4}$.
	
\end{corollary}

	\medskip 
	We refer the interested reader to Section \ref{prototypesMM} for further examples of prototypes model manifolds to which Theorem \ref{mainthm-1} applies and the interpretation of the corresponding results. Next we turn to the statement of our second order results. By combining spherical harmonics decomposition technique with Theorem \ref{mainthm-1} we derive the following second order analogue of Theorem \ref{mainthm-1}
	
	\begin{theorem}\label{mainthm-2}
		 Let $\,(M,g)$ be an $N$-dimensional Riemannian model with $N\geq 3$ and with metric $g$ as given in \eqref{meetric} with $\psi$ satisfying \eqref{psi}. Then, for all $u \in C_c^{\infty}(M),$ there holds
		\begin{align} \label{eq_mainthm-2}
	\int_{M} (\Delta_{r,g} u)^2 \ {\rm d}v_g &- \frac{(N-1)}{4} \int_{M}\left[ \Lambda_{\pi,r}^{rad}+4(K_{\pi,r}^{rad} - H_{\pi,r}^{tan}) \right] \rgt \ {\rm d}v_g \\& \geq \frac{1}{4} \int_{M} \frac{1}{r^2}\rgt \ {\rm d}v_g +\frac{(N^2- 1)}{4}\int_{M} \frac{1}{\psi^2} \rgt \ {\rm d}v_g\,,\nonumber
	\end{align}
	with $\Lambda_{\pi,r}^{rad} $ as defined in \eqref{lambdarad}.
	
		\end{theorem}
		
	Under suitable curvature bounds inequality \eqref{eq_mainthm-2} implies the following more explicit inequality:	
		
		\begin{corollary}\label{cor_mainthm-2}
	 Let $\,(M,g)$ be an $N$-dimensional Riemannian model with $N\geq 3$ and with metric $g$ as given in \eqref{meetric} with $\psi$ satisfying \eqref{psi} and \eqref{con2}.
	  If furthermore
		\begin{equation}\label{con3}
		K_{\pi,r}^{rad} \geq H_{\pi,r}^{tan}\quad \text{in }(0,+\infty)\,,
		\end{equation} 
		then there holds 
		\begin{align}\label{further}
	\int_{M} (\Delta_{r,g} u)^2 \ {\rm d}v_g &-  \left( \frac{N-1}{2} \right)^{2} \int_{M} \rgt {\rm d}v_g \\& \geq \frac{1}{4} \int_{M} \frac{1}{r^2}\rgt \ {\rm d}v_g +\frac{(N^2- 1)}{4}\int_{M} \frac{1}{\psi^2} \rgt \ {\rm d}v_g\,,\nonumber
	\end{align}
for all $u \in C_c^{\infty}(M)$. 
\end{corollary}

\begin{remark}
	Condition \eqref{con3} holds with the equality if $M=\hn$; examples of models satisfying \eqref{con2} and for which the strict inequality holds in \eqref{con3} can be given by taking $\psi(r)= re^{br}$ for $r$ large, with $b\geq 1$, see the proof of Corollary \ref{hardy-0} in Section \ref{prototypesMM}.
	\end{remark}
	
When $M = \hn$ \eqref{eq_mainthm-2} clearly coincides with \eqref{further} but we also have the optimality of the constant $ \left( \frac{N-1}{2} \right)^{2} $. For the sake of clarity we give the precise statement here below:
		\begin{corollary}\label{use_mainthm_2}
	 Let $M=\hn$, the Hyperbolic space with $N\geq 3$. 	Then, for all $u \in C_c^{\infty}(\hn)$ there holds
	 \begin{align*}
		\int_{\hn} (\Delta_{r,g} u)^2 \ {\rm d}v_{\hn}& -   \left( \frac{N-1}{2} \right)^{2} \int_{\hn} \rgt {\rm d}v_{\hn}\\ &\geq \frac{1}{4} \int_{\hn} \frac{1}{r^2}\rgt {\rm d}v_{\hn} +\frac{(N^2- 1)}{4}\int_{\hn} \frac{1}{(\sinh r)^2} \rgt {\rm d}v_{\hn}
		\end{align*}
		for all $ u \in C_{c}^{\infty}(\hn)$.  Furthermore the constant  $ \left( \frac{N-1}{2} \right)^{2} $ in the above inequality turns out to be sharp in the sense that no inequality of the form 
		$$
		\int_{\hn} (\Delta_{r,g} u)^2  \ {\rm d}v_{\hn}  \geq  c  \int_{\hn} \rgt \  {\rm d}v_{\hn}
		$$
		holds, for  $ u \in C_{c}^{\infty}(\hn)$ when $c >  \left( \frac{N-1}{2} \right)^{2}.$
	
\end{corollary}

	\medskip
	Next we state a Rellich type improvement for the second order Poincar\'e inequality \eqref{higher_order_poin} with $\ell=0$ but on more general model manifolds.  The proof comes by exploiting either inequality \eqref{base} on Riemannian models (see Lemma \ref{rad_lap} in the following) and inequality \eqref{eq_mainthm-2}, the first brought the restriction $N\geq 5$ and the latter yields the curvature conditions  \eqref{con2} and \eqref{con3} below.

	\begin{theorem}\label{mainthm-3}
	  	 Let $\,(M,g)$ be an $N$-dimensional Riemannian model with $N\geq 5$ and with metric $g$ as given in \eqref{meetric} with $\psi$ satisfying \eqref{psi},  \eqref{con2} and \eqref{con3}. Then for all $ u \in C_{c}^{\infty}(M)$ there holds

	  	\begin{align}\label{eq_mainthm-3}
	  	\int_{M} (\Delta_{g} u)^2 \ {\rm d}v_{g} -   \left( \frac{N-1}{2} \right)^{4}
	  	\int_{M} u^2 \ {\rm d}v_{g} \geq \frac{(N-4)^2}{16} \int_{M} \frac{u^2}{r^4} \ {\rm d}v_{g} + 
	  	\frac{(N-1)^2}{16} \int_{M} \frac{u^2}{r^2} {\rm d}v_{g}\,.
	  	\end{align}
	\end{theorem}

	\begin{remark}
As already explained in the Introduction, inequality  \eqref{eq_mainthm-3} with $M=\hn$ must be compared with inequality \eqref{PR}. In particular, it gives rise to the interesting fact that the constant appearing in front of the Rellich term $\frac{u^2}{r^4}$ can be larger than $\frac{9}{16}$.
	\end{remark}

	\medskip
	We conclude by stating a Hardy-type improvement of the second order Poincar\'e inequality \eqref{higher_order_poin} with $\ell=1$ on model manifolds.  Here the main tools exploited in the proofs are spherical harmonics decomposition and reduction of dimension technique. The latter yields the restriction $N\geq 5$, while conditions  \eqref{con2} and \eqref{con3} come again from inequality \eqref{eq_mainthm-2} that we apply for each component of the decomposition.

	\begin{theorem}\label{mainthm-4}
		Let $\,(M,g)$ be an $N$-dimensional Riemannian model with $N\geq 5$ and with metric $g$ as given in \eqref{meetric} with $\psi$ satisfying \eqref{psi}, 
		\eqref{con2} and \eqref{con3}.
		Then for all $ u \in C_{c}^{\infty}(M)$ there holds
		\begin{equation}\label{eq_mainthm-4}
			\int_{M} (\Delta_{g} u)^2 \ {\rm d}v_g -   \left( \frac{N-1}{2} \right)^{2}  \int_{M} |\nabla_{g} u|^2 \ {\rm d}v_g  \\  \geq \frac{1}{4} \int_{M} \frac{|\nabla_{g} u|^2}{r^2} \ {\rm d}v_g +\frac{(N^2- 1)}{4}\int_{M} \frac{|\nabla_{g}u|^2}{\psi^2} \ {\rm d}v_g\,.
		\end{equation}
		
	\end{theorem}

\section{Some prototype model manifolds} \label{prototypesMM}
In this section we discuss our first order results on $N$-dimensional Riemannian models $\,(M,g)$ with $N\geq 3$ and with metric $g$ as given in \eqref{meetric} with $\psi$ satisfying condition \eqref{psi} and the further condition for large $r$:
\begin{equation} \label{prototype1}
\psi(r)=A e^{br^{a+1}} \quad \text{ for } r \geq R
\end{equation}
or
\begin{equation} \label{prototype2}
 \psi(r)=A re^{br^{a+1}} \quad \text{ for } r \geq R
\end{equation}
for some $R,A,b>0$ and $a\geq -1$. The case \eqref{prototype1} includes $\hn$ (for $a=0$), while the case \eqref{prototype2} includes $\rn$ (for $a=-1$). In both cases the sectional curvatures satisfy
$$K_{\pi,r}^{rad}\sim -b^2 (a+1)^2 r^{2a} \quad \text{ and } \quad H_{\pi,r}^{rad}\sim -b^2 (a+1)^2 r^{2a}\quad  \text{as } r \rightarrow +\infty \,.$$
Hence, for $a\geq 0$ unbounded curvatures from below are allowed. We refer to \cite[Section 2.3]{gmv} for further possible choices of $\psi$ and their geometric interpretation.
\medskip
\par

 In case \eqref{prototype1}  from Theorem \ref{mainthm-1} we derive the following improved Poincar\'e inequality for functions supported outside $B_R(o)$:

\begin{corollary}\label{hardy-00}
	 Let $\,(M,g)$ be an $N$-dimensional Riemannian model with $N\geq 3$ and with metric $g$ as given in \eqref{meetric} with $\psi$ satisfying conditions \eqref{psi} and \eqref{prototype1} for some $R,A,b>0$ and $a\geq 0$. Then, for all  $u \in C_c^{\infty}(M \setminus B_R(o)),$ there holds
	
			\begin{align}\label{proto}
\int_{M} \left(\frac{\partial u}{\partial r} \right)^2 \, {\rm d}v_g 
& \geq \left(\frac{N-1}{2} \right)^2(a+1)^2 b^2 \int_{M} r^{2a} \, u^2 \, {\rm d}v_g + \frac{1}{4} \int_{M} \frac{u^2}{r^2} \, {\rm d}v_g  \\ \notag
& + \frac{2ba(a+1)(N-1)}{4} \int_{M} \, r^{a-1} u^2 \, {\rm d}v_g.
\end{align}		
	 \end{corollary}
	Notice that \eqref{proto} can be seen as an improved Poincar\'e inequality since $\int_{M} r^{2a} \, u^2 \, {\rm d}v_g\geq \int_{M} u^2 \, {\rm d}v_g$ for $a\geq 0$ and $u \in C_c^{\infty}(M \setminus B_R(o))$ with $R\geq 1$, for instance. In particular, for $a=0$ and $b=1$ we recover the sharp Poincar\'e constant in $\hn$, i.e. $\left(\frac{N-1}{2} \right)^2$.
 \begin{proof}
 For $r \geq R$ we compute:
	
		$$
		K_{\pi,r}^{rad}= -b^2 (a+1)^2 r^{2a} - b a(a+1) r^{a-1}\,,\quad 
		H_{\pi,r}^{tan}= - b^2 (a+1)^2 r^{2a} + \frac{1}{A^2 e^{2br^{a+1}}} 
		$$
		and
		$$\Lambda_{\pi,r}^{rad}=2ba(a+1)r^{a-1}+b^2 (a+1)^2(N-1) r^{2a}-\frac{N-3}{ A^2e^{2br^{a+1}}} \,.$$
		Then, the desired inequality comes by inserting the above term into \eqref{eq_mainthm-1}, summing up and rearranging all terms.

 \end{proof}

In case \eqref{prototype2}  from Theorem \ref{mainthm-1} we derive an improved Hardy inequality for functions supported outside $B_R(o)$:

\begin{corollary}\label{hardy-0}
	 Let $\,(M,g)$ be an $N$-dimensional Riemannian model with $N\geq 3$ and with metric $g$ as given in \eqref{meetric} with $\psi$ satisfying conditions \eqref{psi} and \eqref{prototype2} for some $R,A,b>0$ and $a\geq -1$. Then, for all  $u \in C_c^{\infty}(M \setminus B_R(o)),$ there holds
	
			\begin{align}\label{proto2}
\int_{M} \left(\frac{\partial u}{\partial r} \right)^2 \, {\rm d}v_g 
& \geq \frac{(N-2)^2}{4} \int_{M} \frac{u^2}{r^2} \, {\rm d}v_{g} +  \frac{(a+1)^2 b^2 (N-1)^2}{4} \int_{M} r^{2a} \, u^2 \, {\rm d}v_g   \\ \notag
& + \frac{b(a+1)(N-1)(N-1+a)}{2} \int_{M} \, r^{a-1} u^2 \, {\rm d}v_g.
\end{align}		
	 \end{corollary}
	 When $a=-1$ the decay of the manifold is of euclidean type and \eqref{proto2} reduces to the standard Hardy inequality. On the other hand, when $a\geq 0$, since $\int_{M} r^{2a} \, u^2 \, {\rm d}v_g\geq \int_{M} u^2 \, {\rm d}v_g$ for all $u \in C_c^{\infty}(M \setminus B_R(o))$ with $R\geq 1$, \eqref{proto2} reads as an improved Hardy-Poincar\'e inequality.
 \begin{proof}
 For $r \geq R$ we compute:
	
		$$
		K_{\pi,r}^{rad} = -b^2 (a+1)^2 r^{2a} - b (a+1)(a+2) r^{a-1}\,,\quad 
		H_{\pi,r}^{tan} = -b^2 (a+1)^2 r^{2a}- 2b (a+1) r^{a-1}- \frac{1}{r^2} + \frac{1}{A^2 r^2e^{2br^{a+1}}} 
		$$
		
		and
		$$\Lambda_{\pi,r}^{rad}=2b(N-1+a)r^{a-1}+b^2 (a+1)^2(N-1) r^{2a}+\frac{N-3}{r^2}-\frac{N-3}{A^2 r^2e^{2br^{a+1}}}\,.$$
		Then, the \eqref{proto2}  comes by inserting the above term into \eqref{eq_mainthm-1}, summing up and rearranging all terms.

 \end{proof}

A further remarkable consequence of Theorem \ref{mainthm-1} is the following improved Hardy inequality:

\begin{corollary}\label{hardy-2}
	 Let $\,(M,g)$ be an $N$-dimensional Riemannian model with $N\geq 3$ and with metric $g$ as given in \eqref{meetric} with $\psi(r)=r e^{{r^{2m}}}$ for some positive integer $m$.
Then, for all  $ u \in C_c^{\infty}(M \setminus\{0\})$ there holds
		\begin{equation}\label{hardy-1}
			\begin{split}
				\int_{M} \rgt \, {\rm d}v_{g} \geq \frac{(N-2)^2}{4} \int_{M} \frac{u^2}{r^2} \, {\rm d}v_{g} 
				+m^2 (N-1)^2 \int_{M}  r^{4m-2}\,u^2 \, {\rm d}v_{g}
				\\+m (N-1)(N-2+2m) \int_{M}  r^{2m-2}\,u^2 \, {\rm d}v_{g} \,,
			\end{split}
		\end{equation}
		where the constant $\frac{(N-2)^2}{4}$ is sharp.
	 \end{corollary}
	 \begin{remark}
	 It's worth noticing that, under the assumption of Corollary \ref{hardy-2}, we have:
	 	 $$\frac{(N-1)}{4} \Lambda_{\pi,r}^{rad}=\frac{(N-1)(N-3)}{4} \frac{1}{r^2}- \frac{(N-1)(N-3)}{4} \frac{1}{\psi^2(r)}+f_{m,N}(r) \,,$$
		 for some $f_{m,N}(r)  >0$.
	 Once this observed, we readily infer that, as stated in Theorem \ref{mainthm-1}, the constant in front of the term $\frac{1}{r^2}$ on the right hand side of \eqref {eq_mainthm-1} cannot be larger than $\frac{1}{4}$ otherwise we would contradict the sharpness of the Hardy constant $\frac{(N-2)^2}{4}$.
	 \end{remark}
	 \begin{proof}
	 Clearly, the function $\psi(r)=r e^{{r^{2m}}}$ satisfies condition \eqref{psi}, hence Theorem \ref{mainthm-1} applies. On the other hand, some computations yield:
		
		$$
		K_{\pi,r}^{rad}=-(2m)^2 r^{4m-2} -2m(2m+1) r^{2m-2}\,,\quad
		H_{\pi,r}^{tan}=-(2m)^2 r^{4m-2}- \frac{1}{r^2} -4m r^{2m-2} + \frac{1}{r^2 e^{2r^{2m}}}\,,
		$$
		and
		$$\Lambda_{\pi,r}^{rad}=(2m)^2(N-1)  r^{4m-2} +4m(N-2+2m) r^{2m-2}+\frac{(N-3)}{r^2}- \frac{(N-3)}{r^2 e^{2r^{2m}}}\,.$$
		Finally, inequality \eqref{hardy-1} comes by inserting the above term into \eqref{eq_mainthm-1}, summing up and rearranging all terms. As for the sharpness of the constant $\frac{(N-2)^2}{4}$, it comes from the results in \cite{Carron}.

	 \end{proof}

	\medskip

	\section{Proofs of Theorem~\ref{mainthm-1}, Corollary \ref{use_cor_1} and Corollary \ref{use_cor_1}}\label{3}
	
	\subsection{Proof of Theorem~\ref{mainthm-1}}
	This section is devoted to the proof of Theorem~\ref{mainthm-1}. Before starting we recall the following useful lemma that can be easily derived from the Sturm-Comparison Theorem, see \cite[Lemma 2.1]{PRS} and that we will repeatedly exploit in our proofs.
		\begin{lemma}\label{sturm}
		 Let $\,(M,g)$ be an $N$-dimensional Riemannian model with metric $g$ as given in \eqref{meetric} with $\psi$ satisfying \eqref{psi}.
		If \eqref{con2} holds, then
		\begin{align}\label{sturm_comp}
			\frac{\psi^{\prime}(r)}{\psi(r)} \geq \coth r \quad  \text{and}\quad \psi(r) \geq \sinh r \quad \forall\, r>0\,.
		\end{align}

	\end{lemma}
	
Next we discuss some basic facts on spherical harmonics. More details can be found in \cite{MSS}. 
	
	\medskip 	

	{\bf Spherical harmonics.}
	
	Let $u(x)=u(r,\sigma)\in C_c^\infty(M)$, $r\in[{0},\infty)$ and $\sigma\in \mathbb{S}^{N-1}$, we can write
	
	\begin{equation*}
		u(x):=u(r,\sigma)=\sum_{n=0}^{\infty}d_n(r)P_n(\sigma)
	\end{equation*}
	
	in $L^2(M)$, where $\{ P_n \}$ is an orthonormal system of spherical harmonics and 
	
	\begin{equation*}
		d_n(r)=\int_{\mathbb{S}^{N-1}}u(r,\sigma)P_n(\sigma) \ {\rm d}\sigma\,.
	\end{equation*} 
	
	A spherical harmonic $P_n$ of order $n$ is the restriction to $\mathbb{S}^{N-1}$ of a homogeneous harmonic polynomial of degree $n.$ Moreover, it
	satisfies $$-\Delta_{\mathbb{S}^{N-1}}P_n=\lambda_n P_n$$
	for all $n\in\mathbb{N}_0$ where $\lambda_n=(n^2+(N-2)n)$ are the eigenvalues of Laplace Beltrami operator $-\Delta_{\mathbb{S}^{N-1}}$ on $\mathbb{S}^{N-1}$ with corresponding eigenspace dimension $c_n$. We note that $\lambda_n\geq 0$, $\lambda_0=0$, $c_0=1$, $c_1=N$ and 
	
	\begin{equation*}
		c_n=\binom{N+n-1}{n}-\binom{N+n-3}{n-2}
	\end{equation*}
	
	for $n\geq 2$. 
	
		\medskip 
	
	We are now ready to begin the proof. 
	
	\medskip 
	
	{\bf Proof of \eqref{eq_mainthm-1}.}  The proof is divided into two steps.
	
	\medskip 
{\bf Step 1. }
 Let $u\in C_c^\infty(M \setminus \{ o \}),$ where $o$ denotes the pole, and define the transformation
\begin{equation*}
v(x)=\psi(r)^{\frac{(N-1)}{2}}u(x) \text{ where } r=\rho(x,o) \text{ and } x=(r,\sigma)\in (0,\infty)\times \mathbb{S}^{N-1}\,.
\end{equation*}
Then clearly $v \in C_c^{\infty}(M \setminus \{ o \}).$

\medskip

 An easy calculation gives 

\begin{align*}
\frac{\partial v }{\partial r} = \frac{\partial }{\partial r}\bigg(\psi(r)^{\frac{(N-1)}{2}}u \bigg)= \psi(r)^{\frac{(N-1)}{2}}\frac{\partial u}{\partial r} + \frac{(N-1)}{2}\frac{\psi'(r)}{\psi(r)}\:v\,.
\end{align*}

By arranging the terms we obtain:
 
\begin{equation}\label{Th-2.1_step_1}
\frac{\partial u }{\partial r}=\frac{1}{\psi(r)^{\frac{(N-1)}{2}}}\bigg[\frac{\partial v }{\partial r} -\frac{(N-1)}{2}\frac{\psi'(r)}{\psi(r)}\:v\bigg]\,.
\end{equation}

\medskip

{\bf Step 2.}   	

Now, expanding $v$ in terms of spherical harmonics:
$$v(x)=v(r,\sigma)=\sum_{n=0}^{\infty}d_n(r)P_n(\sigma)\,,$$
we find
$$\frac{\partial v }{\partial r}=\sum_{n=0}^{\infty}d'_n(r)P_n(\sigma)\,.$$

Furthermore, from \eqref{Th-2.1_step_1} we observe that

\begin{align}\label{Th-2.1_step_2_0}
\int_{M} \rgt \ {\rm d}v_{g} &= \int_{M}\frac{1}{\psi ^{(N-1)}}\bigg(\frac{\partial v}{\partial r}\bigg)^2 \ {\rm d}v_{g}-(N-1)\int_{M}\frac{1}{\psi^{(N-1)}}\frac{\partial v }{\partial r} \frac{\psi'}{\psi}\:v \ {\rm d}v_{g}\nonumber\\&+ \frac{(N-1)^2}{4}\int_{M}\frac{1}{\psi^{(N-1)}}\frac{\psi'^2}{\psi^2}\: v^2 \ {\rm d}v_{g}\,.
\end{align} 

\medskip

We will evaluate each term of \eqref{Th-2.1_step_2_0} separately. Using the integration by parts formula and the orthonormal properties of $\{P_n\}$ i.e $\int_{\mathbb{S}^{N-1}}P_nP_m\:{\rm d}\sigma=\delta_{nm}$, we find

\begin{align}\label{Th-2.1_step_2_1}
\int_{M}\frac{1}{{\psi}^{(N-1)}}\bigg(\frac{\partial v}{\partial r}\bigg)^2 \ {\rm d}v_{g} &= \sum_{n=0}^{\infty}\int_{0}^{\infty}(d'_n)^2 \ {\rm d}r
\end{align}

\begin{align}\label{Th-2.1_step_2_2}
\int_{M}\frac{1}{\psi^{(N-1)}}\frac{\partial v }{\partial r} \frac{\psi'}{\psi}\:v \ {\rm d}v_{g}  =-\frac{1}{2}\sum_{n=0}^{\infty}\int_{0}^{\infty}d_n^2\:\frac{\psi\psi''-\psi'^2}{\psi^2} \ {\rm d}r
\end{align}

\begin{align}\label{Th-2.1_step_2_3}
\int_{M}\frac{1}{\psi^{(N-1)}}\frac{\psi'^2}{\psi^2}\: v^2 \ {\rm d}v_{g} &=\sum_{n=0}^{\infty}\int_{0}^{\infty}d_n^2\:\frac{\psi'^2}{\psi^2} \ {\rm d}r\,.
\end{align}
By substituting \eqref{Th-2.1_step_2_1}, \eqref{Th-2.1_step_2_2} and \eqref{Th-2.1_step_2_3} into \eqref{Th-2.1_step_2_0} and after simplifying, we obtain 

\begin{align}\label{Th-2.1_step_2_4}
\int_{M} \rgt \ {\rm d}v_{g}&=\sum_{n=0}^{\infty}\int_{0}^{\infty}(d'_n)^2 \ {\rm d}r+\frac{(N-1)}{2}\sum_{n=0}^{\infty}\int_{0}^{\infty}d_n^2\,\frac{\psi\psi''-\psi'^2}{\psi^2} \ {\rm d}r\\ \notag
&+\frac{(N-1)^2}{4}\sum_{n=0}^{\infty}\int_{0}^{\infty}d_n^2\,\frac{\psi'^2}{\psi^2} \ {\rm d}r\nonumber\\&=\sum_{n=0}^{\infty}\int_{0}^{\infty}(d'_n)^2 \ {\rm d}r -\frac{(N-1)}{4}\sum_{n=0}^{\infty}\int_{0}^{\infty}
\big[2K_{\pi,r}^{rad}+(N-3)H_{\pi,r}^{tan}\big]d_n^2 \ {\rm d}r\notag \\ &+ \frac{(N-1)(N-3)}{4}\sum_{n=0}^{\infty}\int_{0}^{\infty} \frac{d_n^2}{\psi^2} \ {\rm d}r\,.
\end{align}

 Next, by using the $1$-dimensional Hardy inequality:
\begin{equation*}
\int_{0}^{\infty} (d'_n)^2 \ {\rm d}r\geq\frac{1}{4} \int_{0}^{\infty} \frac{d_n^2}{r^2} \ {\rm d}r
\end{equation*}
into \eqref{Th-2.1_step_2_4} we get 
\begin{align}\label{Th-2.1_step_2_5}
\int_{M} \rgt \ {\rm d}v_{g}&\geq  \frac{1}{4}\sum_{n=0}^{\infty}\int_{0}^{\infty} \frac{d_n^2}{r^2} \ {\rm d}r -\frac{(N-1)}{4}\sum_{n=0}^{\infty}\int_{0}^{\infty}
\big[2K_{\pi,r}^{rad}+(N-3)H_{\pi,r}^{tan}\big]d_n^2 \ {\rm d}r\nonumber\\&+\frac{(N-1)(N-3)}{4}\sum_{n=0}^{\infty}\int_{0}^{\infty}\frac{d_n^2}{\psi^2} \ {\rm d}r\,.
\end{align}
Finally, inserting the equalities:
\begin{align*}
\int_{M} \big[2K_{\pi,r}^{rad}+(N-3)H_{\pi,r}^{tan}\big]  u^2 \ {\rm d}v_{g}=\sum_{n=0}^{\infty}\int_{0}^{\infty} \big[2K_{\pi,r}^{rad}+(N-3)H_{\pi,r}^{tan}\big] d_n^2 \ {\rm d}r\,,
\end{align*}
\begin{align*}
\int_{M}\frac{u^2}{r^2} \ {\rm d}v_{g}=\sum_{n=0}^{\infty}\int_{0}^{\infty}\frac{d_n^2}{r^2} \ {\rm d}r\: \text{and}\: \int_{M}\frac{u^2}{\psi^2} \ {\rm d}v_{g}=\sum_{n=0}^{\infty}\int_{0}^{\infty}\frac{d_n^2}{\psi^2} \ {\rm d}r
\end{align*}
 into \eqref{Th-2.1_step_2_5} and recalling \eqref{lambdarad} we obtain the desired inequality \eqref{eq_mainthm-1} with $u \in C_c^{\infty}(M \setminus \{ o \}).$ 
\par
\begin{remark} \label{density}

If \eqref{con2} holds inequality \eqref{eq_mainthm-1} can be extended to functions belonging to $C_c^{\infty}(M)$ by density arguments. Indeed, in this case $M$ is a Cartan-Hadamard manifold with strictly negative curvatures and, since, for $N > 2,$ the set $\{ o\}$ is compact and has zero capacity, the following inclusion holds: $\overline{C_c^{\infty}(M)}^{\|\nabla \cdot \|_2} \subset \overline{C_c^{\infty}(M  \setminus \{ o \})}^{\|\nabla \cdot \|_2}$ (see \cite[Proposition~A.1 and Theorem 6.5]{Dambrosio}).
 Then, by using Gauss Lemma $|\frac{\partial u}{\partial r}| \leq |\nabla u|,$ one deduces the validity of \eqref{eq_mainthm-1} for all $C_c^{\infty}(M)$.
 \end{remark}

\medskip

{\bf Optimality :} We set

\begin{align}\label{define_C_M}
C_{M}:=\text{inf}_{H^1(M)\setminus\{0\}}\:\frac{\int_{M}(\rg)^2 \ {\rm d}v_{g} - \frac{(N-1)}{4}\int_{M}\Lambda_{\pi,r}^{rad}\,u^2 \ {\rm d}v_{g}}{\int_{M}\frac{u^2}{r^2} \ {\rm d}v_g}\,.
\end{align} 

If \eqref{con2} holds, then $\Lambda_{\pi,r}^{rad}>0$ and, by combining density arguments with Fatou's Lemma, we infer that \eqref{eq_mainthm-1} holds in $H^1(M)$ and, in turn, that $C_{M}\geq \frac{1}{4}$. So it remains to show that $C_{M}\leq \frac{1}{4}$ and this will be done by giving a proper minimizing sequence.  Again we divide the proof in some steps.
\medskip

{\bf Step 1.} 
Let us define the sequence $\{\phi_n\}$ for $n\in\mathbb{N}$ as follows

\begin{equation}\label{main_fun}
\phi_n(r)=
\begin{dcases}
0 & 0<r\leq 1 \\
n^{-\alpha}(r-1) & 1\leq r\leq 2 \\
n^{-\alpha} & 2\leq r\leq n \\
r^{-\alpha} & n\leq r \\
\end{dcases}
\end{equation}

where  $\alpha>1+a$ and $a\geq0$ is as given in \eqref{con2bis}. Clearly, $\phi_n\in L^1_{loc}(0,+\infty)$ and its weak derivative writes
\begin{equation}\label{main_fun2}
\phi_n'(r)=
\begin{dcases}
0 & 0<r\leq 1 \\
n^{-\alpha} & 1\leq r\leq 2 \\
0 & 2\leq r\leq n \\
-\alpha r^{-\alpha-1} & n\leq r \,.
\end{dcases}
\end{equation}

Next we recall that, from  the proof of Proposition 4.1 in \cite{EDG}, the function $u_0(r):=\frac{r^\frac{1}{2}}{\psi^\frac{N-1}{2}}$ satisfies:

\begin{align}\label{critic_eq}
-\Delta_{r,g}\:u_0- \frac{(N-1)}{4}\Lambda_{\pi,r}^{rad}\: u_0=\frac{1}{4}\frac{u_0}{r^2}+\frac{(N-1)(N-3)}{4}\:\frac{u_0}{\psi^2}\quad \text{for }r>0\,.
\end{align}

Using polar coordinates and by exploiting \eqref{con2bis}, it follows that $u_0\phi_n$ and $u_0\phi_n^2$ both belong to $H^1(M)$ for all $n\geq 3$ and for $\alpha>1+a$. In particular, \eqref{con2bis} assures $(u_0\phi_n)',(u_0\phi_n^2)'\in L^2(M)$. Indeed, for $r\geq n$, we have that 
\begin{equation}\label{asymta} 
((u_0\phi_n)')^2 \psi^{N-1}\sim \frac{1}{r^{2\alpha-1}}\left( \left(\alpha-\frac{1}{2}\right)^2 \frac{1}{r^2}+\frac{(N-1)^2}{4}\,C^2\,r^{2a} + (2\alpha-1) \frac{(N-1)}{2}\,C\, r^{a-1}  \right)
\end{equation}

and this term turns out to be integrable at infinity for $\alpha>1+a$, if $\psi$ satisfies the above condition; the term $(u_0\phi_n^2)'$ can be treated similarly.
\medskip

\medskip

{\bf Step 2.}
Let $R>n$, by multiplying the equation \eqref{critic_eq} by $u_0 \phi_n^2$ and integrating by parts, we obtain
\begin{align}\label{useful_eqR}
&\int_{B_R(0)}\bigg( \frac{\partial u_0}{\partial r} \bigg) \bigg( \frac{\partial (u_0\phi_n^2) }{\partial r}\bigg) \ {\rm d}v_{g}- \int_{\partial B_R(0)} \bigg( \frac{\partial u_0}{\partial r} \bigg)\, u_0\phi_n^2  \ {\rm d}S_{g}
- \frac{(N-1)}{4}\int_{B_R(0)}\Lambda_{\pi,r}^{rad} \, (u_0 \phi_n)^2 \ {\rm d}v_{g}\\ \notag
&=\frac{1}{4}\int_{B_R(0)}\frac{(u_0  \phi_n)^2}{r^2} \ {\rm d}v_{g}+\frac{(N-1)(N-3)}{4}\int_{B_R(0)}\frac{(u_0 \phi_n)^2}{\psi^2} \ {\rm d}v_{g}\,,
\end{align}
 where we have exploited the fact that $\phi_n$ is supported outside $B_1(0)$, hence no problem of integrability arises at $r=0$.

Next we note that
\begin{align*}
\int_{\partial B_R(0)} \bigg( \frac{\partial u_0}{\partial r} \bigg)\, u_0\phi_n^2  \ {\rm d}S_{g}&=\int_{\partial B_R(0)} u'_0(R)\, u_0(R)\phi_n^2(R)  \ {\rm d}S_{g}\\
&=\frac{1}{2}\left( 1-(N-1)R \frac{\psi'(R)}{\psi(R)}\right) \frac{1}{R^{\alpha}}=o(1) \quad \text{as } R \rightarrow + \infty
\end{align*}
where in the above we have exploited the fact that $|\partial B_R(0)|=\omega_N\,\psi^{N-1}(R)$. \par
On the other hand, for $r>n$, we have that

\begin{align*}
&\big|\bigg( \frac{\partial u_0}{\partial r} \bigg) \bigg( \frac{\partial (u_0\phi_n^2) }{\partial r}\bigg)\big|=\big|(u_0')^2 \phi_n^2+2u_0u_0' \phi_n \phi'_n\big|\\
&=\bigg|\frac{1}{r^{2\alpha-1}}\bigg(\frac{1}{4}\frac{1}{r^2}+\frac{(N-1)^2}{4}\bigg(\frac{\psi'}{\psi}\bigg)^2+\bigg( \alpha  -\frac{N-1}{2}\bigg)\frac{\psi'}{\psi}\frac{1}{r} -\alpha\frac{1}{r^2}\bigg)\frac{1}{\psi^{N-1}}\bigg| \\
&\leq \frac{1}{r^{2\alpha-1}}\bigg(\frac{1}{4}\frac{1}{r^2}+\frac{(N-1)^2}{4} \, C^2\, r^{2a}+\bigg( \alpha +\frac{N-1}{2}\bigg)\, C \,r^{a-1} +\alpha\frac{1}{r^2}\bigg)\frac{1}{\psi^{N-1}}\quad \text{as } r \rightarrow + \infty\text{ (using  \eqref{con2bis})}
 \\ \\
&\frac{(u_0  \phi_n)^2}{r^2}= \frac{1}{r^{2\alpha+1}}\frac{1}{\psi^{N-1}}
\\ \\
&\frac{(u_0 \phi_n)^2}{\psi^2}\leq \frac{1}{r^{2\alpha+1}}\frac{1}{\psi^{N-1}} \quad \text{(using \eqref{con2} and \eqref{sturm_comp}). } 
\end{align*}

As for the terms involving the curvatures, we first note that, if \eqref{con2} holds then $\Lambda_{\pi,r}^{rad} >0$ (see the proof of Corollary \ref{use_cor_1} below) and, by density arguments, \eqref{eq_mainthm-1} yields
$$\frac{(N-1)}{4} \int_{M}  \Lambda_{\pi,r}^{rad}  \,  u^2 \, {\rm d}v_{g}\leq \int_{M} \rgt \, {\rm d}v_{g}  \quad \text{for all } u\in H^1(M)\,.$$
Hence, since $u_0 \phi_n \in H^1(M)$, we deduce that
$$\frac{(N-1)}{4} \int_{M} \Lambda_{\pi,r}^{rad} \, (u_0 \phi_n)^2 \, {\rm d}v_{g} \leq \int_{M} \left(\frac{\partial (u_0 \phi_n) }{\partial r}\right)^2 \, {\rm d}v_{g} $$
where the boundedness of the latter term and, in turn, of the first, follows by \eqref{asymta}.
\par
Since for all terms listed above we have integrability at infinity whenever $\alpha>1+a$, by Lebesgue Theorem, we can pass to the limit in \eqref{useful_eqR} as $R \rightarrow + \infty$ and we conclude that

\begin{align}\label{useful_eq}
\int_{M} \bigg( \frac{\partial u_0}{\partial r} \bigg) \bigg( \frac{\partial (u_0\phi_n^2) }{\partial r}\bigg) \ {\rm d}v_{g}-\frac{(N-1)}{4}\int_{M}\Lambda_{\pi,r}^{rad} \, (u_0 \phi_n)^2 \ {\rm d}v_{g}\\ \notag
=\frac{1}{4}\int_{M}\frac{(u_0  \phi_n)^2}{r^2} \ {\rm d}v_{g}+\frac{(N-1)(N-3)}{4}\int_{M}\frac{(u_0 \phi_n)^2}{\psi^2} \ {\rm d}v_{g}\,.
\end{align}

\medskip

{\bf Step 3.} Next we observe that 

\begin{align*}
\int_{M} \bigg( \frac{\partial u_0}{\partial r} \bigg) \bigg( \frac{\partial (u_0\phi_n^2) }{\partial r}\bigg) \ {\rm d}v_{g}=\int_{M}\bigg( \frac{\partial (u_0\phi_n)}{\partial r}\bigg)^2 \ {\rm d}v_{g} - \int_{M}u_0^2 \bigg( \frac{\partial \phi_n}{\partial r}\bigg)^2 \ {\rm d}v_{g}\,.
\end{align*}
 
Note that, by the above discussion, the first two terms are well defined. Furthermore, for $r\geq n$, we have that
$$u_0^2 \bigg( \frac{\partial \phi_n}{\partial r}\bigg)^2\psi^{N-1}=\frac{\alpha^2}{r^{2\alpha+1}}$$ 
which is in fact integrable for $\alpha>0$. \par

By using this into \eqref{useful_eq} we have

\begin{align}
\int_{M}\bigg( \frac{\partial (u_0\phi_n)}{\partial r}\bigg)^2 \ {\rm d}v_{g} &- \frac{(N-1)}{4}\int_{M}\Lambda_{\pi,r}^{rad} \, (u_0\phi_n)^2 \ {\rm d}v_{g}& \nonumber\\&=\frac{1}{4}\int_{M}\frac{(u_0\phi_n)^2}{r^2} \ {\rm d}v_{g}+\frac{(N-1)(N-3)}{4}\int_{M}\frac{(u_0\phi_n)^2}{\psi^2} \ {\rm d}v_{g}+ \int_{M}u_0^2 \bigg( \frac{\partial \phi_n}{\partial r}\bigg)^2 \ {\rm d}v_{g} \nonumber
\end{align}

and, by considering the quotient in \eqref{define_C_M} with $u=u_0\phi_n$, we obtain

\begin{align*}
\frac{\int_{M}u_0^2 \big( \frac{\partial (u_0\phi_n)}{\partial r}\big)^2 \ {\rm d}v_{g} \ {\rm d}v_{g} - \frac{(N-1)}{4}\int_{M} \Lambda_{\pi,r}^{rad}  \, (u_0\phi_n)^2 \ {\rm d}v_{g}}{\int_{M}\frac{(u_0\phi_n)^2}{r^2} \ {\rm d}v_g}\nonumber \\=\frac{1}{4}+\frac{(N-1)(N-3)}{4}\frac{\int_{M}\frac{(u_0\phi_n)^2}{\psi^2} \ {\rm d}v_{g}}{\int_{M}\frac{(u_0\phi_n)^2}{r^2} \ {\rm d}v_g}+\frac{\int_{M}u_0^2 \big( \frac{\partial \phi_n}{\partial r}\big)^2 \ {\rm d}v_{g}}{\int_{M}\frac{(u_0\phi_n)^2}{r^2} \ {\rm d}v_g}\,.
\end{align*}

Then, from definition of $C_M$ and the above, we infer 

\begin{align}\label{quotient}
C_{M}\leq \frac{1}{4}+\frac{(N-1)(N-3)}{4}\frac{\int_{M}\frac{(u_0\phi_n)^2}{\psi^2} \ {\rm d}v_{g}}{\int_{M}\frac{(u_0\phi_n)^2}{r^2} \ {\rm d}v_g}+\frac{\int_{M}u_0^2 \big( \frac{\partial \phi_n}{\partial r}\big)^2 \ {\rm d}v_{g}}{\int_{M}\frac{(u_0\phi_n)^2}{r^2} \ {\rm d}v_g}\,.
\end{align}

\medskip

{\bf Step 4.} We estimate each term of the r.h.s. of \eqref{quotient}. Note that $\omega_N$ denotes the $N$ dimensional measure of unit sphere, hence a finite number.

First we estimate the denominator
\begin{align}\label{denom_est}
\int_{M}\frac{(u_0\phi_n)^2}{r^2} \ {\rm d}v_{g}&=\omega_N\int_{0}^{\infty}\frac{(u_0\phi_n)^2}{r^2}(\psi(r))^{(N-1)} \ {\rm d}r=\omega_N\int_{0}^{\infty}\frac{\phi_n^2}{r} \ {\rm d}r\nonumber\\&\geq\omega_N\int_{2}^{n}\frac{\phi_n^2}{r} \ {\rm d}r+\omega_N\int_{n}^{\infty}\frac{\phi_n^2}{r} \ {\rm d}r\nonumber\\&=\omega_N\:n^{-2\alpha}\bigg[\log(\frac{n}{2})+\frac{1}{2\alpha}\bigg]\,.
\end{align} 

Now we consider
\begin{align}\label{numa_3_est}
\int_{M}u_0^2 \left( \frac{\partial \phi_n}{\partial r}\right)^2 \ {\rm d}v_{g}&=\omega_N\int_{0}^{\infty}u_0^2(\phi_n')^2(\psi(r))^{(N-1)} \ {\rm d}r=\omega_N\int_{0}^{\infty}r(\phi_n')^2 \ {\rm d}r\nonumber\\&=\omega_N\int_{1}^{2}r(\phi_n')^2 \ {\rm d}r+\omega_N\int_{n}^{\infty}r(\phi_n')^2 \ {\rm d}r\nonumber\\&=\omega_N\:n^{-2\alpha}\bigg[\frac{3}{2}+\frac{\alpha}{2}\bigg]\,.
\end{align}

Finally, using integration by parts, $\alpha>1$, \eqref{sturm_comp}  and $\sinh r\geq r$, we obtain

\begin{align}\label{numa_2_est}
\int_{M}\frac{(u_0\phi_n)^2}{\psi^2} \ {\rm d}v_{g}&=\omega_N\int_{0}^{\infty}\frac{(u_0\phi_n)^2}{\psi^2(r)}(\psi(r))^{(N-1)} \ {\rm d}r=\omega_N\int_{0}^{\infty}\frac{r(\phi_n)^2}{\psi^2(r)} \ {\rm d}r\nonumber\\&=\omega_N\int_{1}^{2}\frac{r(\phi_n)^2}{\psi^2(r)} \ {\rm d}r+\omega_N\int_{2}^{n}\frac{r(\phi_n)^2}{\psi^2(r)} \ {\rm d}r+\omega_N\int_{n}^{\infty}\frac{r(\phi_n)^2}{\psi^2(r)} \ {\rm d}r\nonumber\\&\leq\omega_N\int_{1}^{2}\frac{r\:n^{-2\alpha}(r-1)^2}{\text{sinh}^2\:r} \ {\rm d}r+\omega_N\int_{2}^{n}\frac{r\:n^{-2\alpha}}{\text{sinh}^2\:r} \ {\rm d}r+\omega_N\int_{n}^{\infty}\frac{r^{1-2\alpha}}{\text{sinh}^2\:r} \ {\rm d}r\nonumber\\& \leq \frac{\omega_N\:n^{-2\alpha}}{\text{sinh}^2\:1}\int_{1}^{2}r(r-1)^2 \ {\rm d}r+\omega_N\:n^{-2\alpha}\int_{2}^{n}(\text{csch}\:r) \ {\rm d}r +\omega_N\int_{n}^{\infty}r^{-1-2\alpha} \ {\rm d}r\nonumber\\& = \omega_N\:n^{-2\alpha} \bigg[\frac{7}{12\:(\text{sinh}^2\:1)}+\log\bigg|\frac{\exp(-n)-1}{\exp(-n)+1}\bigg|-\log\bigg|\frac{\exp(-2)-1}{\exp(-2)+1}\bigg|+\frac{1}{\alpha}\bigg]\,.
\end{align}

\medskip

{\bf Step 5.} We are now in the final stage. Using \eqref{numa_2_est} and \eqref{denom_est} we have
\begin{align*}
\frac{\int_{M}\frac{(u_0\phi_n)^2}{\psi^2} \ {\rm d}v_{g}}{\int_{M}\frac{(u_0\phi_n)^2}{r^2} \ {\rm d}v_g}\leq\frac{\frac{7}{12\:(\text{sinh}^2\:1)}+\log\bigg|\frac{\exp(-n)-1}{\exp(-n)+1}\bigg|-\log\bigg|\frac{\exp(-2)-1}{\exp(-2)+1}\bigg|+\frac{1}{2\alpha}}{\log(\frac{n}{2})+\frac{1}{2\alpha}} \rightarrow 0\quad \text{for}\: n \rightarrow\infty
\end{align*}

and, using \eqref{numa_3_est} and \eqref{denom_est}, we deduce that

\begin{align*}
\frac{\int_{M}u_0^2 \big( \frac{\partial \phi_n}{\partial r}\big)^2 \ {\rm d}v_{g}}{\int_{M}\frac{(u_0\phi_n)^2}{r^2} \ {\rm d}v_g}\leq\frac{\frac{3}{2}+\frac{\alpha}{2}}{\log(\frac{n}{2})+\frac{1}{2\alpha}}\rightarrow 0\quad \text{for}\: n \rightarrow\infty. 
\end{align*}

Finally, combining these two and \eqref{quotient}, we can say that,

\begin{align*}
C_{M}\leq \frac{1}{4}+\frac{(N-1)(N-3)}{4}\frac{\int_{M}\frac{(u_0\phi_n)^2}{\psi^2} \ {\rm d}v_{g}}{\int_{M}\frac{(u_0\phi_n)^2}{r^2} \ {\rm d}v_g}+\frac{\int_{M}u_0^2 \big( \frac{\partial \phi_n}{\partial r}\big)^2 \ {\rm d}v_{g}}{\int_{M}\frac{(u_0\phi_n)^2}{r^2} \ {\rm d}v_g}=\frac{1}{4}+o(1)\,.
\end{align*}

Hence, $\{u_0\phi_n\}$  is the required minimizing sequence and, in turn, $C_{M}=\frac{1}{4}$.

\medskip

\subsection{Proof of Corollary \ref{use_cor_1}}

	Recalling \eqref{eq_mainthm-1}, we know that
	
	\begin{equation*}
	\int_{M} \rgt \ {\rm d}v_g -   \frac{N-1}{4} \int_{M} \Lambda_{\pi,r}^{rad}  \, u^2 \ {\rm d}v_g   \geq \frac{1}{4} \int_{M} \frac{u^2}{r^2} \ {\rm d}v_g +\frac{(N-1)(N-3)}{4}\int_{M} \frac{u^2}{\psi^2} \ {\rm d}v_g,
	\end{equation*}
	for all $u\in C_c^\infty(M)$. Therefore, to prove the statement, it is enough to show that 
	$  \Lambda_{\pi,r}^{rad} \geq (N-1)$ for all $r>0$. This follows by using \eqref{con2} and \eqref{sturm_comp}, by which we deduce that

		$$\Lambda_{\pi,r}^{rad}(r)= \left[ 2 \frac{\psi^{\prime \prime}(r)}{\psi(r)} + (N-3) \frac{((\psi^{{\prime}}(r))^2 - 1)}{\psi^2(r)} \right] \geq \left( 2 + (N-3) \right) = N-1$$
 for all $r>0$. $\qed$

\subsection{Proof of Corollary \ref{use_cor_2}}

	Let $M=\hn$, the statement of Corollary \ref{use_cor_2} comes from Theorem \ref{eq_mainthm-1} by taking $\psi(r)=\sinh(r)$, once proved the sharpness of the constants $ \left( \frac{N-1}{2} \right)^{2}$ and $\frac{(N-1)(N-3)}{4}$. This issue readily follows from the sharpness of the same constants in inequality \eqref{poin_hardy_sharp} otherwise, by Gauss's Lemma, we would get a contradiction. $\qed$

	\medskip
	
	\section{Proofs of Theorem~\ref{mainthm-2} and Corollary \ref{use_mainthm_2}} \label{4}
	\subsection{Proof of Theorem~\ref{mainthm-2}.} 
	For any $u\in C_c^\infty(M)$ we have 
	
	\begin{equation*}
	\Delta_{r,g}u = \frac{\partial^2 u}{\partial r^2}+(N-1)\frac{\psi^{\prime}}{\psi} \frac{\partial u}{\partial r}\,.
	\end{equation*}
	
	Exploiting the spherical harmonic decomposition, we write
		$$u(x):= u(r, \sigma) = \sum_{n = 0}^{\infty} a_{n}(r) P_{n}(\sigma)\,$$
	and we obtain
	\begin{align*}
	(\Delta_{r,g} u)^2  = \bigg(\sum_{n = 0}^{\infty} \left( a_{n}^{\prime \prime} + (N-1) \frac{\psi^{\prime}}{\psi} a_{n}^{\prime} \right) P_{n}\bigg)^2 \text{ and } \left(\rg \right)^2 =  \left(\sum_{n = 0}^{\infty} a_{n}^{\prime} P_{n}\right)^2.
	\end{align*}

	Recalling that $\{P_n\}$ is orthonormal and by using polar coordinates and integrating by parts, we get
	
	\begin{align*}
	&\int_{M} (\Delta_{r,g} u)^2 \ {\rm d}v_g - \frac{(N-1)}{4} \int_{M} \Lambda_{\pi,r}^{rad} \, \rgt \ {\rm d}v_g  \\
	& = \omega_N \sum_{n = 0}^{\infty} \int_{0}^{\infty}\bigg[(a_{n}^{\prime \prime})^2 - (N-1) (a_n^{\prime})^2\left(\frac{\psi^{\prime}}{\psi}\right)^{\prime} -\frac{(N-1)}{4}  \Lambda_{\pi,r}^{rad} \, (a_n^{\prime})^2\bigg] \psi^{N-1} \ {\rm d}r\,.
	\end{align*}

Now, exploiting Theorem \ref{mainthm-1} for each $a_n^{\prime}$, we get
	\begin{align*}
	&\int_{0}^{\infty}(a_{n}^{\prime \prime})^2\,\psi^{N-1} \ {\rm d}r \geq\frac{(N-1)}{4} \int_{0}^{\infty} \, \Lambda_{\pi,r}^{rad}  (a_{n}^{\prime})^2\,\psi^{N-1} \ {\rm d}r \\
	&+ \frac{1}{4} \int_{0}^{\infty}\frac{(a_{n}^{\prime})^2}{r^2}\,\psi^{N-1} \ {\rm d}r+\frac{(N-1)(N-3)}{4} \int_{0}^{\infty}\frac{(a_{n}^{\prime})^2}{\psi^2}\,\psi^{N-1} \ {\rm d}r
\end{align*}
and, in turn, we infer
	
	\begin{align*}
	& \int_{M} (\Delta_{r,g} u)^2 \ {\rm d}v_g  - \frac{(N-1)}{4} \int_{M} \Lambda_{\pi,r}^{rad} \, \rgt \ {\rm d}v_g  \geq \frac{1}{4} \int_{M} \frac{1}{r^2}\rgt \ {\rm d}v_g \nonumber \\ & +\frac{(N^2- 1)}{4}\int_{M} \frac{1}{\psi^2} \rgt \ {\rm d}v_g +(N-1)\int_M\:\big[K_{\pi,r}^{rad} - H_{\pi,r}^{tan}\big]   \rgt  \ {\rm d}v_g\,.
	\end{align*}
	
	By rearranging we obtain the desired result. $\qed$
	\medskip

\subsection{Proof of Corollary \ref{cor_mainthm-2}.}
The desired inequality follows at once by \eqref{eq_mainthm-2} simply noting that, when \eqref{con2} and \eqref{con3} holds, then
$$ \Lambda_{\pi,r}^{rad}+4(K_{\pi,r}^{rad} - H_{\pi,r}^{tan}) \geq N-1\quad \text{for all } r>0\,.$$
See also the proof of Corollary \ref{use_cor_1}. $\qed$

\subsection{Proof of Corollary \ref{use_mainthm_2}.}

		When  $M = \hn$ all assumptions of Theorem \ref{mainthm-2} are satisfied and since $K_{\pi,r}^{rad} =H_{\pi,r}^{tan}=-1$, the desired inequality follows at once. Concerning the sharpness of the constant  $ \left( \frac{N-1}{2} \right)^{2} $, it follows by combining \eqref{higher_order_poin} with \eqref{base} and Gauss's lemma. $\qed$
		\par
		
		\medskip
\par

For future purposes, we conclude the section by stating a further inequality that follows directly by combining Theorem \ref{mainthm-2} and Corollary \ref{use_cor_1}:
  
  \begin{corollary}\label{use_cor_3}
  	Let $ N \geq 5$ and $M$ as given in \eqref{meetric} satisfying  \eqref{con2} and \eqref{con3}.
  	Then, for all $ u \in C_{c}^{\infty}(M)$ there holds
  	
  	\begin{align}\label{eq_use_cor_3}
  	\int_{M} (\Delta_{r,g} u)^2 \ {\rm d}v_g & - \,
  	\left( \frac{N-1}{2} \right)^{4}  \int_{M} u^2  \, {\rm d}v_{g}
  	\geq \frac{(N-1)^2}{16} \int_{M} \frac{u^2}{r^2} \, {\rm d}v_{g} + \frac{(N-1)^3(N-3)}{16} \int_{M} \frac{u^2}{\psi^2} \, {\rm d}v_{g} \nonumber \\ & + \frac{1}{4} \int_{M} \frac{1}{r^2}\rgt \ {\rm d}v_g +\frac{(N^2- 1)}{4}\int_{M} \frac{1}{\psi^2} \rgt \ {\rm d}v_g\,.
  	\end{align}
  \end{corollary}

	\section{Proof of Theorem~\ref{mainthm-3}}\label{5}
	
	A key ingredient in the proof of Theorem~\ref{mainthm-3} will be Lemma \ref{rad_lap} below which is proved buying the lines of  \cite[Theorem 5.2]{VHN} where the same inequality is given in $\hn$. Here we extend its statement to more general manifolds using Sturm Comparison Principle, i.e. our Lemma \ref{sturm}.
	
	\begin{lemma}\label{rad_lap}
	Let $\,(M,g)$ be an $N$-dimensional Riemannian model with metric $g$ as given in \eqref{meetric} with $\psi$ satisfying \eqref{psi} and \eqref{con2}. Then, for $0\leq \beta < N-4$, there holds 
		\begin{equation}\label{eq_rad_lap}
		\int_{M}(\Delta_{g} u)^2r^{-\beta}\: {\rm d}v_{g} \geq \int_{M}(\Delta_{r,g} u)^2r^{-\beta}\: {\rm d}v_{g}\quad   \text{   for all  }u\in C_c^\infty(M).
		\end{equation}
		Moreover, the equality holds when $u$ is a radial function.
		
	\end{lemma}

	\begin{proof}
		
		Consider any $u\in C_c^\infty(M)$, then
		
		\begin{align*}
		\Delta_{r,g}u = \frac{\partial^2 u}{\partial r^2}+(N-1)\frac{\psi^{\prime}}{\psi} \frac{\partial u}{\partial r}\quad \text{ and } \quad
		\Delta_{g} u = \frac{\partial^2 u}{\partial r^2}+(N-1)\frac{\psi^{\prime}}{\psi}\frac{\partial u}{\partial r}+\frac{1}{\psi^2}\Delta_{S^{N-1}}\,.
		\end{align*}
		
		Now, by decomposing into spherical harmonics the above expressions, see the proof of Theorem \ref{mainthm-2} for more details, we write 
			$$u(x):= u(r, \sigma) = \sum_{n = 0}^{\infty} a_{n}(r) P_{n}(\sigma)\,.$$
		
		From this decomposition of $u$ we have
		\begin{equation*}
		\Delta_{r,g} u=\sum_{n=0}^{\infty}\Delta_{g} a_n(r)P_n(\sigma)
		\end{equation*}
		and 
		\begin{equation*}
		\Delta_{g} u=\sum_{n=0}^{\infty}\bigg(\Delta_{g} a_n(r)-\lambda_n \frac{a_n(r)}{\psi(r)^2}\bigg)P_n(\sigma)\,.
		\end{equation*}
		
		Hence, to prove \eqref{eq_rad_lap} it is enough to show that 
		\begin{align*}
		\lambda_n\int_M\frac{a_n^2}{r^{\beta}\psi^4}\: {\rm d}v_g - 2 \int_M
		\frac{a_n(\Delta_g a_n)}{r^{\beta}\psi^2}\: {\rm d}v_g\geq 0\quad \text{  for all } n\geq 1.
		\end{align*}
		
		Since we have that $2a_n(\Delta_ga_n)=\Delta_g(a_n^2)-2|\nabla_ga_n|^2$, using by parts formula, the above inequality is equivalent to
		
		\begin{align}\label{sub_step}
		\lambda_n\int_M\frac{a_n^2}{r^{\beta}\psi^4}\: {\rm d}v_g - \int_M
		a_n^2\Delta_g \bigg(\frac{1}{r^{\beta}\psi^2}\bigg)\: {\rm d}v_g+2\int_M
		\frac{|\nabla_g a_n|^2}{r^{\beta}\psi^2}\: {\rm d}v_g\geq 0\text{  for all } n\geq 1.
		\end{align}
		
		Set $\kappa(r):=\frac{1}{r^{\beta}\psi^2}$,  by exploiting \eqref{con2}, we have 
		\begin{align}\label{sub_step_key_1}
		-\Delta_g\kappa(r)&=\kappa(r)\bigg(2 \frac{\psi''}{\psi} +2(N-4)\frac{{\psi^{\prime}}^2}{\psi^2}-\frac{\beta(\beta+1)}{r^2}+\beta(N-5)\frac{\psi^{\prime
		}}{r\psi}\bigg)\\  \notag
		& \geq  \kappa(r)\bigg(2 +2(N-4)\frac{{\psi^{\prime}}^2}{\psi^2}-\frac{\beta(\beta+1)}{r^2}+\beta(N-5)\frac{\psi^{\prime}}{r\psi}\bigg)\,.
		\end{align}
		
		We claim that for $f\in C_c^\infty (M)$ radial there holds
		\begin{align}\label{sub_step_key_2}
		\int_M\frac{|\nabla_g f|^2}{r^{\beta}\psi^2}\: {\rm d}v_g\geq \frac{(N-\beta-4)^2}{4}\int_M\frac{f^2}{r^{\beta}\psi^4}\: {\rm d}v_g
		\end{align}
		
		which, in polar coordinate, is equivalent to
		
		\begin{align}\label{main_step}
		\int_{0}^{\infty}{f^{\prime}}^2r^{-\beta}\psi^{N-3}\: {\rm d}r\geq \frac{(N-\beta-4)^2}{4}\int_{0}^{\infty}{f}^2r^{-\beta}\psi^{N-5}\:  {\rm d}r\,.
		\end{align}
		
		Integrating by parts we have the following
		
		\begin{align}\label{use_eqn}
		&-\int_{0}^{\infty}{ff^{\prime}}r^{N-\beta-4}\big(\frac{\psi}{r}\big)^{N-5}\: {\rm d}r=\\ \notag
		&\frac{(N-\beta-4)}{2}\int_{0}^{\infty}{f}^2r^{-\beta}\psi^{N-5}\: {\rm d}r+\frac{(N-5)}{2}\int_{0}^{\infty}{f}^2r^{-\beta}\psi^{N-4} 
		 \left(\frac{\psi' r-\psi}{\psi^2}\right) \: {\rm d}r\,.
		\end{align}
		
		On the other hand, using Young's inequality and Lemma \ref{sturm} by which $\frac{r}{\psi} \leq 1$ for all $r>0$, we have 
		\begin{align*}
		-\int_{0}^{\infty}   f f'  \,r^{N-\beta-4}\big(\frac{\psi}{r}\big)^{N-5}\: {\rm d}r =-\int_{0}^{\infty} (f r^{-\beta/2}\psi^{(N-5)/2})(f' r^{-\beta/2}\psi^{(N-3)/2} \frac{r}{\psi})\: {\rm d}r \\ 
		\leq \frac{(N-\beta-4)}{4}\int_{0}^{\infty}{f}^2r^{-\beta}\psi^{N-5}\:  {\rm d}r+\frac{1}{(N-\beta-4)}\int_{0}^{\infty}{f^{\prime}}^2r^{-\beta}\psi^{N-3}\: {\rm d}r\,
		\end{align*}
		
		and, by combining the above inequality with \eqref{use_eqn} and recalling that by Lemma \ref{sturm} there holds $
		 \left(\frac{\psi' r-\psi}{\psi^2}\right) \geq 0$ for all $r>0$, we obtain \eqref{main_step}.
		
		With the help of \eqref{sub_step_key_1} and \eqref{sub_step_key_2} into \eqref{sub_step}, it is enough to prove
		
		\begin{align*}
		\int_M \frac{a_n^2}{r^\beta\psi^4}\bigg[\lambda_n+\frac{(N-\beta-4)^2}{2} +2\psi^2+2(N-4){\psi^{\prime}}^2-\beta(\beta+1)\frac{\psi^2}{r^2}+\beta(N-5)\frac{\psi^{\prime
			}\psi}{r}\bigg]{\rm d}v_g\geq 0
		\end{align*}
		for all $n\geq 1$, which, in our assumptions, follows by showing that 
		
		\begin{align*}
		\lambda_n+\frac{(N-\beta-4)^2}{2} +2\psi^2+2(N-4){\psi^{\prime}}^2-\beta(\beta+1)\frac{\psi^2}{r^2}+\beta(N-5)\frac{\psi^{\prime
			}\psi}{r}\geq 0\,
		\end{align*}
		for all $r>0$. Indeed, by \eqref{sturm_comp} and $(\coth r)\geq \frac{1}{r}$, we can estimate the above term by below as follows 
		
		\begin{align*}
		\lambda_n+\frac{(N-\beta-4)^2}{2} +2\psi^2+\bigg[2(N-4)-\beta(\beta+1)+\beta(N-5)\bigg]\frac{\psi^2}{r^2}\,.
		\end{align*}
		
		Finally, the conclusion comes by noticing that
		\begin{align*}
		2(N-4)-\beta(\beta+1)+\beta(N-5)\geq 0
		\end{align*}
		for all $0\leq \beta < N-4$.
		
		\end{proof}

{\bf Proof of Theorem~\ref{mainthm-3}.} 

	Let $u\in C_c^\infty(M)$, by combining Corollary \ref{use_cor_3} and Lemma \ref{rad_lap}, we obtain
	
		\begin{align}\label{cor_step_1}
  		\int_{M} (\Delta_{g} u)^2 \ {\rm d}v_g & - \,
  	\left( \frac{N-1}{2} \right)^{4}  \int_{M} u^2  \, {\rm d}v_{g}
  	\geq \frac{(N-1)^2}{16} \int_{M} \frac{u^2}{r^2} \, {\rm d}v_{g} + \frac{(N-1)^3(N-3)}{16} \int_{M} \frac{u^2}{\psi^2} \, {\rm d}v_{g} \nonumber \\ & + \frac{1}{4} \int_{M} \frac{1}{r^2}\rgt \ {\rm d}v_g +\frac{(N^2- 1)}{4}\int_{M} \frac{1}{\psi^2} \rgt \ {\rm d}v_g\,.
  	\end{align}

	Now, from \cite[Theorem 3.1]{VHN}, we know that for all $u\in C_c^\infty(M)$  with $N\geq 5$  there holds
	
	\begin{equation*}
	\int_{M} \frac{1}{r^2}\rgt \, {\rm d}v_{g}  \,
	\geq \frac{(N-4)^2}{4}  \int_{M} \frac{u^2}{ r^4} \, {\rm d}v_{g} \, 
	\end{equation*}
	
	which, substituted into \eqref{cor_step_1}, gives
	
	\begin{align*}
	\int_{M} (\Delta_{g} u)^2 \ {\rm d}v_g  - \,
	\left( \frac{N-1}{2} \right)^{4}  \int_{M} u^2  \, {\rm d}v_{g}
	& \geq \frac{(N-4)^2}{16}  \int_{\hn} \frac{u^2}{ r^4} \, {\rm d}v_{g}  +  \frac{(N-1)^2}{16} \int_{M} \frac{u^2}{r^2} \, {\rm d}v_{g} \nonumber \\ & +\frac{(N-1)^3(N-3)}{16} \int_{M} \frac{u^2}{\psi^2} \, {\rm d}v_{g} + \frac{(N^2- 1)}{4}\int_{M} \frac{1}{\psi^2} \rgt \ {\rm d}v_g
	\end{align*}
	
	and the non-negativity of last two terms immediately gives \eqref{eq_mainthm-3}. This concludes the proof. $\qed$

	\medskip

	\section{Proof of Theorem~\ref{mainthm-4}} \label{6}

	Let $u \in C_{c}^{\infty}(M)$, by spherical harmonics decomposition we have
		$$u(x):= u(r, \sigma) = \sum_{n = 0}^{\infty} a_{n}(r) P_{n}(\sigma)\,,$$
		see again the proof of Theorem \ref{mainthm-2} for more details. Then, 
		$$|\nabla_{g}u|^2 =  \sum_{n = 0}^{\infty} (a_{n}^{\prime})^2 P_{n}^2  + \frac{a_{n}^2}{\psi^2}  |\nabla_{\mathbb{S}^{N-1}}P_{n}|^2$$

	and 
	\begin{align*}
		(\Delta_{g} u)^2 & = \sum_{n = 0}^{\infty} \left( a_{n}^{\prime \prime} + (N-1) \frac{\psi^{\prime}}{\psi} a_{n}^{\prime} \right)^2 P_{n}^2 + 
		\sum_{n = 0}^{\infty} \frac{a_{n}^2}{\psi^4} (\Delta_{\mathbb{S}^{N-1}} P_{n})^2 \\
		& + 2 \sum_{n = 0}^{\infty} \left( a_{n}^{\prime \prime} + (N-1) \frac{\psi^{\prime}}{\psi} a_{n}^{\prime} \right) \frac{a_{n}}{\psi^2} (\Delta_{\mathbb{S}^{N-1}} P_{n}) P_{n}.
	\end{align*}
	
	Let us compute the r.h.s. of \eqref{eq_mainthm-4} in terms of $a_{n}$ and $P_{n}.$ Using the fact that $\int_{\mathbb{S}^{N-1}} P_{n}P_{m} \:d\sigma = \delta_{nm}$, we obtain
	
	\begin{align}\label{Th-2.4_rhs}
		\frac{1}{4} \int_{M} \frac{|\nabla_{g} u|^2}{r^2} \ {\rm d}v_g  & + \frac{(N^2- 1)}{4}\int_{M} \frac{|\nabla_{g}u|^2}{\psi^2} \ {\rm d}v_g = \frac{\omega_N}{4} \sum_{n = 0}^{\infty}  \int_{0}^{\infty} \frac{(a_{n}^{\prime})^{2}}{r^2} \psi^{N-1} \ {\rm d} r \notag \\
		& + \frac{\omega_N}{4} \sum_{n = 0}^{\infty} \lambda_{n} \int_{0}^{\infty} \frac{a_{n}^2}{r^2 \psi^2} \psi^{N-1} \ {\rm d} r + \omega_N\frac{(N^2 - 1)}{4} \sum_{n = 0}^{\infty}  \int_{0}^{\infty} \frac{(a_{n}^{\prime})^{2}}{\psi^2} \psi^{N-1} \ {\rm d} r \notag \\
		& +\omega_N\frac{(N^2 - 1)}{4} \sum_{n = 0}^{\infty} \lambda_{n}  \int_{0}^{\infty} \frac{a_{n}^2}{ \psi^4} \psi^{N-1} \ {\rm d} r.
	\end{align}
	
	Next we consider the l.h.s.of \eqref{eq_mainthm-4}, we have 
	
	\begin{align}\label{Th-2.4_lhs}
		&\int_{M} (\Delta_{g} u)^2 \ {\rm d}v_g  - \left( \frac{N-1}{2} \right)^{2} \int_{M} |\nabla_{g} u|^2 \ {\rm d}v_g \notag \\
		&= \omega_N \sum_{n = 0}^{\infty} \int_{0}^{\infty} \left( \left( a_{n}^{\prime \prime} + (N-1) \frac{\psi^{\prime}}{\psi} a_{n}^{\prime} \right)^2 \right) \psi^{N-1} \ {\rm d}r\nonumber \\
		& +\omega_N \int_{\mathbb{S}^{N-1}} \int_{0}^{\infty} \left( \sum_{n = 0}^{\infty} \frac{a_{n}^2}{\psi^4} (\Delta_{\mathbb{S}^{N-1}} P_{n})^2 \right) \psi^{N-1} \ {\rm d}r{\rm d}\sigma\nonumber \\
		& + 2 \omega_N \int_{\mathbb{S}^{N-1}} \int_{0}^{\infty} \left(\sum_{n = 0}^{\infty} \left( a_{n}^{\prime \prime} + (N-1) \frac{\psi^{\prime}}{\psi} a_{n}^{\prime} \right) \frac{a_{n}}{\psi^2} (\Delta_{\mathbb{S}^{N-1}} P_{n}) P_{n}\right) \psi^{N-1} \ {\rm d}r {\rm d}\sigma \nonumber\\
		& - \omega_N \left( \frac{N-1}{2} \right)^{2} \int_{0}^{\infty} \left(\sum_{n = 0}^{\infty} (a_{n}^{\prime})^2\right) \psi^{N-1} \ {\rm d}r  - 
		\omega_N \left( \frac{N-1}{2} \right)^{2} \int_{0}^{\infty} \left(\sum_{n= 0}^{\infty} \lambda_{n} \frac{a_{n}^2}{\psi^2}\right) \psi^{N-1} \ {\rm d}r.
	\end{align}
	
	We consider each term of the r.h.s. of \eqref{Th-2.4_lhs} separately. 
	
	First we use \eqref{eq_mainthm-2} for each $a_n(r)$ and we get
	\begin{align}\label{Th-2.4_sub_1}
		& \omega_{N} \sum_{n = 0}^{\infty} \int_{0}^{\infty} \left( a_{n}^{\prime \prime} + (N-1) \frac{\psi^{\prime}}{\psi} a_{n}^{\prime} \right)^2 \psi^{N-1} \ {\rm d}r \geq 
		\omega_N \left( \frac{N-1}{2} \right)^{2}  \sum_{n = 0}^{\infty}  \int_{0}^{\infty} (a_{n}^{\prime})^2 \psi^{N-1} \ {\rm d}r \notag  \\
		& + \frac{\omega_{N}}{4}  \sum_{n = 0}^{\infty}  \int_{0}^{\infty} \frac{(a_{n}^{\prime})^2}{r^2}  \psi^{N-1} \ {\rm d}r \,
		+ \omega_{N}\frac{(N^2 - 1)}{4}  \sum_{n = 0}^{\infty}  \int_{0}^{\infty} \frac{(a_{n}^{\prime})^2}{\psi^2}  \psi^{N-1} \ {\rm d}r\,.
	\end{align}
	
	Then, we exploit the equation $-\Delta_{\mathbb{S}^{N-1}}P_n=\lambda_n P_n$, the orthonormal properties of the $\{P_n\}$ and by parts formula to obtain
	
	\begin{align}\label{Th-2.4_sub_2}
		\int_{\mathbb{S}^{N-1}} \int_{0}^{\infty}  \sum_{n = 0}^{\infty} \frac{a_{n}^2}{\psi^4} (\Delta_{\mathbb{S}^{N-1}} P_{n})^2 \psi^{N-1} \ {\rm d}r {\rm d}\sigma =
		\omega_{N} \sum_{n= 0}^{\infty} \lambda_{n}^2 \int_{0}^{\infty} \frac{a_{n}^2}{\psi^4} \psi^{N-1} \ {\rm d}r
	\end{align}
	and
	\begin{align}\label{Th-2.4_sub_3}
		&  2 \int_{\mathbb{S}^{N-1}} \int_{0}^{\infty} \sum_{n = 0}^{\infty} \left( a_{n}^{\prime \prime} + (N-1) \frac{\psi^{\prime}}{\psi} a_{n}^{\prime} \right) \frac{a_{n}}{\psi^2} (\Delta_{\mathbb{S}^{N-1}} P_{n}) P_{n} \psi^{N-1} \ {\rm d}r {\rm d}\sigma \notag \\
		&= - 2\; \omega_{N} \sum_{n = 0}^{\infty} \lambda_{n}  \int_{0}^{\infty}  \left( a_{n}^{\prime \prime} + 
		(N-1) \frac{\psi^{\prime}}{\psi} a_{n}^{\prime} \right) \frac{a_{n}}{\psi^2} \psi^{N-1} \ {\rm d}r \notag \\
		&= -2\:\omega_N\sum_{n=0}^{\infty}\lambda_n\int_{0}^{\infty}a_n^{\prime\prime}\:a_n\psi^{N-3} \ {\rm d}r \:-\omega_N(N-1)\sum_{n=0}^{\infty}\lambda_n\int_{0}^{\infty}(a_n^2)^\prime\:\frac{\psi^\prime}{\psi}\psi^{N-3} \ {\rm d}r\notag \\
		& = 2\:\omega_{N} \sum_{n = 0}^{\infty} \lambda_{n} \int_{0}^{\infty} (a_{n}^{\prime})^2 \psi^{N-3} \ {\rm d}r \; -2\: \omega_N\sum_{n=0}^{\infty}\lambda_n\int_{0}^{\infty}(a_n^2)^\prime\:\frac{\psi^\prime}{\psi}\psi^{N-3} \ {\rm d}r \notag  \\
		& = 2\:\omega_{N} \sum_{n = 0}^{\infty} \lambda_{n} \int_{0}^{\infty} (a_{n}^{\prime})^2 \psi^{N-3} \ {\rm d}r \; + 2\:  \omega_{N} \sum_{n= 0}^{\infty} \lambda_{n} \int_{0}^{\infty} a_{n}^2 \left( \frac{\psi^{\prime \prime}}{\psi} - \frac{(\psi^{\prime})^2}{\psi^2} \right) \psi^{N-3} \ {\rm d}r\notag  \\
		& + 2 (N-3)\: \omega_{N} \sum_{n= 0}^{\infty} \lambda_{n} \int_{0}^{\infty} a_{n}^2 \left( \frac{\psi^{\prime}}{\psi} \right)^2 \psi^{N-3} \ {\rm d}r \notag \\
		& =  2\:\omega_{N} \sum_{n = 0}^{\infty} \lambda_{n} \int_{0}^{\infty} (a_{n}^{\prime})^2 \psi^{N-3} \ {\rm d}r + 
		2\:\omega_{N} \sum_{n= 0}^{\infty} \lambda_{n} \int_{0}^{\infty} a_{n}^2  \frac{\psi^{\prime \prime}}{\psi} \psi^{N-3} \ {\rm d}r \notag  \\
		& + 2 (N-4)\: \omega_{N} \sum_{n= 0}^{\infty} \lambda_{n} \int_{0}^{\infty} a_{n}^2 \left( \frac{\psi^{\prime}}{\psi} \right)^2 \psi^{N-3} \ {\rm d}r\,.
	\end{align}
	
	Now we estimate the first term of \eqref{Th-2.4_sub_3} by using \eqref{eqn_use_cor_1} for $``N-2"$ dimension. Notice that, to this aim, we need $N-2 \geq 3$, i.e. $N\geq 5$. For the remaining two terms we use \eqref{sturm_comp}, \eqref{con2}, $(\coth r)^2=1+\frac{1}{(\sinh r)^2}\geq 1+\frac{1}{\psi^2}$ and $\frac{1}{r^2}\geq\frac{1}{\psi^2}$ for $r>0$,  to obtain
	
	\begin{align}\label{Th-2.4_sub_4}
		&  2 \int_{\mathbb{S}^{N-1}} \int_{0}^{\infty} \sum_{n = 0}^{\infty} \left( a_{n}^{\prime \prime} + (N-1) \frac{\psi^{\prime}}{\psi} a_{n}^{\prime} \right) \frac{a_{n}}{\psi^2} (\Delta_{\mathbb{S}^{N-1}} P_{n}) P_{n} \psi^{N-1} \ {\rm d}r {\rm d}\sigma \notag \\
		& \geq 2\:\omega_N\frac{(N-3)^2}{4} \sum_{n= 0}^{\infty} \lambda_{n} \int_{0}^{\infty} a_{n}^2 \psi^{N-3} \ {\rm d}r 
		+ 2\:\omega_N\frac{1}{4} \sum_{n = 0}^{\infty} \lambda_{n} \int_{0}^{\infty} \frac{a_{n}^2}{r^2} \psi^{N-3} \ {\rm d}r  \notag \\
		& + 2\:\omega_N  \frac{(N-3)(N-5) }{4} \sum_{n = 0}^{\infty} \lambda_{n} \int_{0}^{\infty} \frac{a_{n}^2}{\psi^2} \psi^{N-3} \ {\rm d}r + 
		2\:\omega_N(N - 3) \sum_{n= 0}^{\infty} \lambda_{n} \int_{0}^{\infty} a_{n}^2  \psi^{N-3} {\rm d}r \notag \\ &+ 2(N-4)\:\omega_N   \sum_{n = 0}^{\infty} \lambda_{n} \int_{0}^{\infty} \frac{a_{n}^2}{\psi^2} \psi^{N-3} \ {\rm d}r \notag \\
		& =\omega_N  \frac{(N-3)(N + 1)}{2} \sum_{n= 0}^{\infty} \lambda_{n} \int_{0}^{\infty} \frac{a_{n}^2}{\psi^2} \psi^{N-1} \ {\rm d}r 
		+ \frac{\omega_N}{4} \sum_{n = 0}^{\infty} \lambda_{n} \int_{0}^{\infty} \frac{a_{n}^2}{r^2 \psi^2} \psi^{N-1} \ {\rm d}r \notag \\
		& +  \omega_N  \frac{(2N^2-8N-1)}{4} \sum_{n = 0}^{\infty} \lambda_{n} \int_{0}^{\infty} \frac{a_{n}^2}{\psi^4} \psi^{N-1} \ {\rm d}r.
	\end{align}
	
	Therefore, taking into account \eqref{Th-2.4_sub_1}, \eqref{Th-2.4_sub_2}, \eqref{Th-2.4_sub_4} and using the fact that $\lambda_n\geq(N-1)=\lambda_{1}$ into  \eqref{Th-2.4_lhs}, we infer that
	
	\begin{align*}
		& \int_{M} (\Delta_{g} u)^2 \ {\rm d}v_g  -\left( \frac{N-1}{2} \right)^{2} \int_{M} |\nabla_{g} u|^2 \ {\rm d}v_g \geq
		 \frac{\omega_{N}}{4} \sum_{n = 0}^{\infty} \int_{0}^{\infty} \frac{(a_{n}^{\prime})^2}{r^2} \psi^{N - 1} \ {\rm d}r 
		 \\ &  + \frac{\omega_{N}}{4} \sum_{n = 0}^{\infty}  \lambda_{n} \int_{0}^{\infty} \frac{a_{n}^2}{r^2 \psi^2} \psi^{N - 1} \ {\rm d}r + \omega_N\frac{(N^2 - 1 )}{4} \sum_{n = 0}^{\infty} \int_{0}^{\infty} \frac{(a_{n}^{\prime})^2}{\psi^2} \psi^{N - 1} \ {\rm d}r \notag  \\
		& +  \omega_N \frac{(2N^2 -4N -5)}{4} \sum_{n = 0}^{\infty}  \lambda_{n} \int_{0}^{\infty} \frac{a_{n}^2}{\psi^4} \psi^{N - 1} \ {\rm d}r \\
		& + \omega_N \bigg[\frac{{(N-3)(N+1)}}{2}-\left( \frac{N-1}{2} \right)^{2}\bigg] \sum_{n = 0}^{\infty}  \lambda_{n} \int_{0}^{\infty} \frac{a_{n}^2}{\psi^2} \psi^{N - 1} \ {\rm d}r.
	\end{align*}

	Hence, by noticing that $\bigg[\frac{{(N-3)(N+1)}}{2}-\left( \frac{N-1}{2} \right)^{2}\bigg]\geq 0$ and $\frac{(2N^2 -4N -5)}{4} \geq \frac{(N^2-1)}{4}$ for $N\geq 5$, and by combining the above inequality with \eqref{Th-2.4_rhs}, we conclude that
	\begin{align*}
		\int_{M} (\Delta_{g} u)^2 \ {\rm d}v_g  -  \left( \frac{N-1}{2} \right)^{2} \int_{M} |\nabla_{g} u|^2 \ {\rm d}v_g 
		\geq \frac{1}{4} \int_{M} \frac{|\nabla_{g} u|^2}{r^2} \ {\rm d}v_g   + \frac{(N^2- 1)}{4}\int_{M} \frac{|\nabla_{g}u|^2}{\psi^2} \ {\rm d}v_g\,,
	\end{align*}
	which is the thesis. $\qed$
	\medskip
	
	\medskip

   \par\bigskip\noindent
\textbf{Acknowledgments.}  The authors are grateful to the anonymous referee whose suggestions helped in improving the preliminary form of the manuscript. The first author is member of the Gruppo Nazionale per l'Analisi Matematica, la Probabilit\`a e le loro Applicazioni (GNAMPA) of the Istituto Nazionale di Alta Matematica (INdAM) and she is partially supported by the PRIN project ``Direct and inverse problems for partial differential equations: theoretical aspects and applications'' (Italy). 
The second author is partially supported by INSPIRE faculty fellowship (IFA17-MA98).
The third author is supported by Council of Scientific \& Industrial Research (File no.  09/936(0182)/2017-EMR-I) and by the PhD program at Indian Institute of Science Education and Research, Pune.



\begin{thebibliography}{99}
  \bibitem{ANR}   A.~Adimurthi,  N.~Chaudhuri, M.~Ramaswamy, \emph{An improved Hardy-Sobolev inequality and its application}, 
Proc. Amer. Math. Soc. 130 (2002), no. 2, 489--505.
  
  \bibitem{AK} K.~Akutagawa, H. Kumura, \emph{Geometric relative Hardy inequalities and the discrete spectrum
of Schrodinger operators on manifolds}, Calc. Var. Partial Differential Equations 48 (2013), 67--88.

\bibitem{GFT} G. Barbatis, S. Filippas, A. Tertikas, \emph{A unified approach to improved $L^p$ Hardy inequalities with best constants},
Trans. Amer. Soc, 356 (2004), 2169-2196.

\bibitem{EDG} E. Berchio, D. Ganguly, G. Grillo, \emph{Sharp Poincar\'e-Hardy and Poincar\'e-Rellich inequalities on the hyperbolic space}, J. Funct. Anal.  272 (2017), 1661--1703.

\bibitem{BG1} E. Berchio, D. Ganguly,  \emph{ Improved higher order Poincar\'e inequalities on the hyperbolic space via Hardy-type remainder terms}, Commun. on Pure and Appl. Analysis 15  (2016), 1871--1892.

\bibitem{BAGG} E. Berchio,  L. D'Ambrosio,  D. Ganguly, G. Grillo, \emph{Improved L$^p$-Poincar\'e inequalities on the hyperbolic space}, Nonlinear Anal. 157 (2017), 146--166. 

\bibitem{BGGP} E. Berchio, D.~Ganguly, G.~Grillo, Y.~Pinchover,  \emph{An optimal improvement for the Hardy inequality
on the hyperbolic space and related manifolds}, to appear in Proc. Roy. Soc. Edinburgh Sect. A (2019).

\bibitem{Brezis} H. Brezis, J. L. Vazquez, \emph{Blow-up solutions of some nonlinear elliptic problems}, Rev. Mat. Univ. Complut. Madrid 10 (1997), 443--469.

\bibitem{Carron} G.~Carron, \emph{Inegalites de Hardy sur les varietes Riemanniennes non-compactes}, J. Math. Pures Appl. (9) 76 (1997), 883--891.

\bibitem{Dambrosio} L. D'Ambrosio, S. Dipierro, \emph{Hardy inequalities on Riemannian manifolds and applications},  Ann. Inst. H. Poinc. Anal. Non Lin. 31 (2014),  449-475.

\bibitem{pinch} B. Devyver, M. Fraas, Y. Pinchover, \emph{Optimal Hardy weight for second-order elliptic operator: an answer to a problem of Agmon}, J. Funct. Anal. 266 (2014), 4422-4489.

\bibitem{RGR} R. Greene, W. Wu, \emph{Function Theory of Manifolds which Possess a Pole}, Springer, 1979.

\bibitem{AG} A. Grigoryan, \emph{Analytic and geometric background of recurrence and non-explosion of the Brownian motion on Riemannian manifolds}, Bull. Amer. Math. Soc. 36 (1999), 135-249.


\bibitem{gmv} G. Grillo, M. Muratori, J.L. V\'azquez, \emph{The porous medium equation on Riemannian manifolds with negative curvature. The large-time behaviour}, Adv. Math. 314 (2017), 328-377.

\bibitem{Kombe1} I. Kombe, M. Ozaydin, \emph{Improved Hardy and Rellich inequalities on Riemannian manifolds}, Trans. Amer. Math. Soc. 361 (2009), no. 12, 6191-6203.

\bibitem{Kombe2} I. Kombe, M. Ozaydin, \emph{Rellich and uncertainty principle inequalities on Riemannian manifolds}, Trans. Amer. Math. Soc. 365 (2013), no. 10, 5035-5050.

\bibitem{AKR} A.~Krist\'aly, \emph{Sharp uncertainty principles on Riemannian manifolds: the influence of curvature}, J. Math. Pures Appl. (9) 119 (2018), 326-346.

\bibitem{KS} A. Krist\'aly, A. Szak\'al, \emph{Interpolation between Brezis-Vazquez and Poincar\'e inequalities on nonnegatively curved spaces: sharpness and rigidities}, J. Diff. Eq. 266 (2019), no. 10, 6621-6646.

\bibitem{PP} P. Petersen; \emph{Riemannian Geometry}, Graduate texts in Mathematics, 171, NY: Springer.xvi, (1998).

\bibitem{Yang} Q. Yang, D. Su, Y. Kong, \emph{Hardy inequalities on Riemannian manifolds with negative curvature}, Commun. Contemp. Math. 16 (2014), no. 2, 1350043.

\bibitem{ngo} Q.A. Ngo, V.H. Nguyen, \emph{Sharp constant for Poincar\'e-type inequalities in the hyperbolic space}, Acta Math. Vietnam. 44 (2019), no. 3, 781-795

\bibitem{VHN} V.H. Nguyen, \emph{New sharp Hardy and Rellich type inequalities on Cartan-Hadamard manifolds and their improvements}, to appear in Proc. Roy. Soc. Edinburgh Sect. A.

\bibitem{MSS} G.~Metafune, M.~Sobajima, C.~Spina, \emph{Weighted Calder\'on Zygmund 
and Rellich inequalities in $L^p$}, Math.Ann. 361 (2015), 313--366.

\bibitem{PRS} S.~Pigola, M.~Rigoli,  A.~G.~Setti, \emph{Vanishing and finiteness results in geometric analysis. A generalization of the Bochner technique}, Progress in Mathematics, 266. Birkhauser Verlag, Basel, 2008. xiv+282 pp. ISBN: 978-3-7643-8641-2.

\bibitem{SK} K. Sandeep, D. Karmakar, \emph{ Adams inequality on the hyperbolic space}, J. Funct. Anal. 270(2016), 1792--1817. 


 \end{thebibliography}
   \end{document}